\documentclass[11pt]{article}
\usepackage[margin=1in]{geometry}
\usepackage{amsmath,amsthm,amsfonts,amssymb}
\usepackage{mathrsfs}
\usepackage{mathtools}
\usepackage{graphicx,enumerate}
\graphicspath{ {./figures/} }
\usepackage[title]{appendix}
\usepackage[
            CJKbookmarks=true,
            bookmarksnumbered=true,
            bookmarksopen=true,
            colorlinks=true,
            citecolor=red,
            linkcolor=blue,
            anchorcolor=red,
            urlcolor=blue
            ]{hyperref}
\mathtoolsset{showonlyrefs}
\usepackage[font=small,labelfont=bf]{caption} 

\newcommand{\Exp}{\mathbb{E}}

\renewcommand{\P}{\mathbb{P}}

\newcommand{\Bin}{\mathsf{Bin}}

\usepackage{tikz}
\usepackage{tikz-qtree}

\usepackage{tikz-cd}

\usetikzlibrary{calc,arrows,cd}
\usetikzlibrary{decorations.pathreplacing,decorations.markings}

\tikzstyle{dot}=[circle,fill,black,inner sep=1pt]

\tikzset{
  on each segment/.style={
    decorate,
    decoration={
      show path construction,
      moveto code={},
      lineto code={
        \path [#1]
        (\tikzinputsegmentfirst) -- (\tikzinputsegmentlast);
      },
      curveto code={
        \path [#1] (\tikzinputsegmentfirst)
        .. controls
        (\tikzinputsegmentsupporta) and (\tikzinputsegmentsupportb)
        ..
        (\tikzinputsegmentlast);
      },
      closepath code={
        \path [#1]
        (\tikzinputsegmentfirst) -- (\tikzinputsegmentlast);
      },
    },
  },
  mid arrow/.style={postaction={decorate,decoration={
        markings,
        mark=at position .5 with {\arrow[#1]{stealth}}
      }}},
  early arrow/.style={postaction={decorate,decoration={
        markings,
        mark=at position .2 with {\arrow[#1]{stealth}}
      }}},
}

\tikzstyle{int}=[draw, fill=blue!15, minimum size=2em]
\tikzstyle{init} = [pin edge={to-,thin,black}]

\def\alternatecolorred{%
    \pgfkeysalso{red}%
    \global\let\alternatecolor\alternatecolorblue 
}
\def\alternatecolorblue{%
    \pgfkeysalso{blue}%
    \global\let\alternatecolor\alternatecolorred 
}

\newcommand{\altred}{\let\alternatecolor\alternatecolorred 
\tikzset{every edge/.append code = {%
    \global\let\currenttarget\tikztotarget 
    \pgfkeysalso{append after command={(\currenttarget)}}
			\alternatecolor
}}
}
\newcommand{\altblue}{\let\alternatecolor\alternatecolorblue 
\tikzset{every edge/.append code = {%
    \global\let\currenttarget\tikztotarget 
    \pgfkeysalso{append after command={(\currenttarget)}}
			\alternatecolor
}}
}

\tikzstyle{vertexdot}=[circle, draw, fill=black, minimum size=3,inner sep=0pt]

\newtheorem{theorem}{Theorem}
\newtheorem{lemma}{Lemma}[section]
\newtheorem{proposition}{Proposition}[section]

\theoremstyle{definition}
\newtheorem{remark}{Remark}
\newtheorem{definition}{Definition}

\usepackage{color}
\usepackage{subfigure}
\usepackage{algorithm}
\usepackage{algorithmic}

\usepackage{xspace,prettyref}

\newcommand{\prob}[1]{ \mathbb{P}\left( #1 \right) }
\newcommand{\tProb}{{\tilde{\mathbb{P}}}}
\newcommand{\tprob}[1]{{ \tProb\left( #1 \right) }}

\newcommand{\Var}{\mathsf{Var}}

\newrefformat{eq}{(\ref{#1})}
\newrefformat{chap}{Chapter~\ref{#1}}
\newrefformat{sec}{Section~\ref{#1}}
\newrefformat{alg}{Algorithm~\ref{#1}}
\newrefformat{fig}{Fig.~\ref{#1}}
\newrefformat{tab}{Table~\ref{#1}}
\newrefformat{rmk}{Remark~\ref{#1}}
\newrefformat{clm}{Claim~\ref{#1}}
\newrefformat{def}{Definition~\ref{#1}}
\newrefformat{cor}{Corollary~\ref{#1}}
\newrefformat{lmm}{Lemma~\ref{#1}}
\newrefformat{prop}{Proposition~\ref{#1}}
\newrefformat{app}{Appendix~\ref{#1}}
\newrefformat{hyp}{Hypothesis~\ref{#1}}
\newrefformat{thm}{Theorem~\ref{#1}}
\newrefformat{ass}{Assumption~\ref{#1}}
\newrefformat{conj}{Conjecture~\ref{#1}}

\newcommand{\pth}[1]{\left( #1 \right)}

\newcommand{\sth}[1]{\left\{ #1 \right\}}

\newcommand{\indc}[1]{{\mathbf{1}_{\left\{{#1}\right\}}}}

\newcommand{\sfG}{{\mathsf{G}}}

\newcommand{\calB}{{\mathcal{B}}}

\newcommand{\calD}{{\mathcal{D}}}
\newcommand{\calE}{{\mathcal{E}}}

\newcommand{\calI}{{\mathcal{I}}}

\newcommand{\calN}{{\mathcal{N}}}

\DeclareMathAlphabet{\varmathbb}{U}{bbold}{m}{n}

\newcommand{\bd}{{\mathbf{d}}}
\newcommand{\bs}{{\mathbf{s}}}
\newcommand{\bt}{{\mathbf{t}}}
\newcommand{\bD}{{\mathbf{D}}}
\newcommand{\bS}{{\mathbf{S}}}
\newcommand{\bT}{{\mathbf{T}}}

\renewcommand{\tilde}{\widetilde}

\usepackage{bbm}

\newcommand{\R}{\mathbb{R}}

\newcommand{\N}{\mathbb{N}}
\newcommand{\Z}{\mathbb{Z}}

\newcommand{\bm}{\boldsymbol}

\newcommand{\ER}{Erd\H{o}s-R\'{e}nyi\xspace}

\newcommand{\SBM}{\mathsf{SBM}}
\newcommand{\dd}{{\mathrm{d}}}
\newcommand{\Vp}[1]{ V_{+}^{(#1)} }
\newcommand{\Vm}[1]{ V_{-}^{(#1)} }
\newcommand{\WW}[2]{ W_{#1}^{(#2)} }

\newcommand{\bfW}{{\mathbf{W}}}

\newcommand{\scrP}{{\mathscr{P}}}
\newcommand{\scrM}{{\mathscr{M}}}
\newcommand{\scrT}{{\mathscr{T}}}
\newcommand{\sP}{{\mathscr{P}}}
\newcommand{\sM}{{\mathscr{M}}}
\newcommand{\sT}{{\mathscr{T}}}
\newcommand{\sN}{{\mathscr{N}}}
\newcommand{\sE}{{\mathscr{E}}}
\newcommand{\sX}{{\mathscr{X}}}

\newcommand{\D}{\Delta'}

\newcommand{\E}{\mathbb{E}}

\newcommand{\F}{\mathcal{F}}

\renewcommand{\P}{\mathbb{P}}

\DeclarePairedDelimiterX{\infdivx}[2]{\Big(}{\Big)}{%
  #1\;\delimsize\|\;#2%
}
\newcommand{\DD}{D\infdivx}

\renewcommand{\DD}[2]{D \pth{\left. #1 \right\| #2 }}

\renewcommand{\ge}{\geqslant}
\renewcommand{\le}{\leqslant}
\renewcommand{\geq}{\geqslant}
\renewcommand{\leq}{\leqslant}

\newcommand{\ee}{\varepsilon}

\begin{document}

\pgfdeclarelayer{background}
\pgfdeclarelayer{foreground}
\pgfsetlayers{background,main,foreground}

\title{Consensus on Dynamic Stochastic Block Models:\\Fast Convergence and Phase Transitions
}

\author{
Haoyu Wang \thanks{Department of Mathematics, Yale University, New Haven, CT 06520, USA,  \texttt{haoyu.wang@yale.edu}}\and
Jiaheng Wei \thanks{Data Science and Analytics Thrust, the Hong Kong University of Science and Technology (Guangzhou), Guangzhou, China,  \texttt{jiahengwei@hkust-gz.edu.cn}}
\and
Zhenyuan Zhang \thanks{Department of Mathematics, Stanford University, Stanford, CA 94305, USA,  \texttt{zzy@stanford.edu}}
}

\date{\today}

\maketitle

\begin{abstract}
We introduce two models of consensus following a majority rule on time-evolving stochastic block models, in which the network evolution is Markovian or non-Markovian.  Under the majority rule, in each round, each agent simultaneously updates their  opinion according to the majority of their  neighbors. In contrast to the classic setting, the dynamics is not purely deterministic and resamples the connections at each step. In the \emph{Markovian model}, connections are resampled at each step and each agent updates their opinion via the majority rule.  We prove a \emph{power-of-one} type result, i.e., any initial bias leads to a non-trivial advantage of winning in the end,  uniformly in the size of the network. 
In the \emph{non-Markovian   model}, a connection is resampled when at least one of them changes opinion and is otherwise kept the same. We identify the phase-transition threshold, up to the
second-order leading term, between halting and fast convergence to consensus.


\end{abstract}



\section{Introduction}

In the theory of distributed computing, consensus refers to the following problem: given a collection of agents holding different opinions, the agents interact and update their opinions under certain rules with the goal of reaching unanimity. One of the most classic and straightforward rules for updating opinions is the \emph{majority dynamics}, where in each round, all agents simultaneously update their opinions based on the majority of their neighbors. Typically, the connections between agents are modeled mathematically using graphs/networks, where vertices represent agents and edges represent connections. Majority dynamics has a long history and allows numerous applications, including economics \cite{ellison1993rules,bala1998learning}, psychology \cite{cartwright1956structural}, biophysics \cite{mcculloch1943logical}, and social choice theory \cite{granovetter1978threshold,nguyen2020dynamics}. See also \cite{mossel2017opinion} for a more recent and detailed account.

Recently there has been surging interest in majority dynamics on random networks.  Among various models of majority dynamics, two classes of formulations are particularly popular and technically tractable. We briefly describe their settings as follows. 

\begin{enumerate}[(a)]
    \item \emph{Majority dynamics on a static random graph}. Consider the \ER graph $\sfG(n,p)$ representing the connections  of the agents, which is fixed throughout the dynamics. Typically the initial opinions held by the agents are assumed to be biased (meaning that  one opinion is held by more people than any of the others), either in a deterministic way \cite{tran2020reaching,sah2021majority,berkowitz2022central} or randomly \cite{benjamini2016convergence,fountoulakis2020resolution,zehmakan2020opinion,chakraborti2021majority}. These opinion processes are non-Markovian at the level of individual vertices, since later updates depend on the previous opinion configuration, but the graph itself has no memory because it is sampled once and then kept fixed. In the dense regime, the recent breakthrough of \cite{sah2021majority} proved a \emph{power-of-one} result, namely any initial bias leads to a nontrivial advantage of winning in the end, uniformly in the size of the network. In sparse regimes, related power-of-few and
    random-initialization results have been established under progressively
    weaker density assumptions \cite{chakraborti2021majority,TV25,kim2025new,jaffe2025new}.
    For more than two opinions, rapid unanimity on $G(n,p)$ has also been proved
    in sparse regimes \cite{CF25}. 
    For other random graph models, see e.g., \cite{gartner2018majority} for random regular graphs and \cite{shang2021note} for inhomogeneous random graphs.

       \item \emph{$k$-majority dynamics}. Consider a (possibly random) graph $G$
    and a fixed integer $k$. In each round, each agent randomly samples $k$
    connections from its neighbors in $G$, with an initial bias on the opinions.
    In the literature, $G$ may refer to the complete graph
    \cite{doerr2011stabilizing,becchetti2016stabilizing,ghaffari2018nearly,
    mukhopadhyay2020voter}, expander graphs as well as \ER graphs
    \cite{cooper2014power,cooper2015fast,cooper2016fast}, and stochastic block
    models \cite{cruciani2019distributed,shimizu2021phase}. In this context,
    the basic unbiased model typically shows that any initial advantage leads
    to consensus. Related biased or noisy communication variants of
    $k$-majority dynamics have also been studied recently
    \cite{CMQR23,DZ25}.
\end{enumerate}

A common feature of these models is the symmetry of connections between agents with or without the same opinion. For example, in majority dynamics on $\sfG(n,p)$, both of the two types of connections are sampled with probability $p$. Nevertheless, there is no reason a priori that an agent is equally likely to draw connections with those holding the same opinion and those with a different opinion. In this paper, we build majority dynamics models beyond the symmetric setting.

To highlight the community structures of those with the same opinion, we introduce two parameters $p,q\in[0,1]$ to represent the connecting probabilities, where $p$ is the probability that two agents with the same opinion are connected, and $q$ is the probability that two agents with different opinions are connected. Motivated by homophily and social-balance phenomena in social networks and social choice theory \cite{cartwright1956structural,granovetter1978threshold,nguyen2020dynamics}, we will assume $p>q$: agents who currently share an opinion are more likely to interact than agents who currently disagree. This structure is captured by the \emph{stochastic block model} (SBM). Introduced by \cite{holland1983stochastic}, SBM is a typical model of inhomogeneous random graphs; see \cite{bollobas2007phase}. The simplest case of an SBM considers a bipartition of the vertices into two blocks, where  edges within a certain block are independently sampled with probability $p$ and otherwise with probability $q$. For more detailed applications in machine learning and computer science, see \cite{abbe2017community}. 

We remark that although there is literature concerning $k$-majority dynamics on SBM, these works do not reflect the correspondence between the blocks and different opinions. The SBM appears there only as a generic prototype of the underlying graph of interest.

\subsection{Models}

To be more precise, our models can be mathematically formulated as follows. Throughout this paper, for simplicity, we consider two opinions, denoted by $+$ and $-$. Without loss of generality, we assume an initial bias with $n+\Delta$ opinions $+$ and $n$ opinions $-$, where $\Delta=\Delta_n$ is a positive integer that may depend on $n$. We are interested in the asymptotic behavior of the model as $n\to\infty$.

\begin{definition}[Majority dynamics]\label{def:Majority_Dynamics}
Given a graph $ G=(V,E) $ whose vertices are indexed by $ [N] :=\{1,\dots,N\} $, each vertex is associated with an initial binary opinion labeled by $ W_i = \WW{i}{0} \in \{\pm 1\} $. Consider a bipartition $ V= \Vp{0} \cup \Vm{0} $, where $ \Vp{0} = \{i \in [N]: \WW{i}{0}=1 \} $ and $ \Vm{0} = \{i \in [N] : \WW{i}{0}=-1 \} $. The \emph{majority dynamics} on $ G $ refers to the following process: At every time step, each vertex updates its opinion based on the majority of its neighbors, i.e.,
\begin{equation}\label{eq:Majority_Rule}
\WW{i}{t+1}=
\left\{
\begin{aligned}
& \mathrm{sign} \pth{\sum_{(i,j)\in E} \WW{j}{t}},& & \mbox{if} \  \sum_{(i,j)\in E} \WW{j}{t} \neq 0 ;\\
& \WW{i}{t}, & & \mbox{otherwise}.
\end{aligned}
\right.
\end{equation}
Let
$$ \Vp{t} := \sth{i \in [N]: \WW{i}{t}=+1},\ \ \mbox{and} \ \  \Vm{t} := \sth{i \in [N]: \WW{i}{t}=-1}.  $$
We say the opinion $ + $ or $ - $ \emph{wins at day $ t $} if $ |\Vp{t}|=N $ or $ |\Vm{t}|=N $, respectively.
\end{definition}

\begin{definition}[Stochastic block models]
For positive integers $ m,n>0 $ and $ p,q \in [0,1] $, the \emph{stochastic block model} $ G \sim \SBM(m,n,p,q) $ is constructed as follows. The graph $ G $ has $ m+n $ vertices, labeled by the elements of $ [m+n]  $. Each vertex $ i \in [m+n] $ has a community label $ W_i \in \{-1,+1\} $. The two communities $ V_+:=\{i \in [m+n]:W_i=+1\} $ and $ V_- := \{i \in [m+n]:W_i =-1\} $ satisfy $ |V_+|=m $ and $ |V_-|=n $. For distinct $ i,j \in [m+n] $, if $ W_i W_j=1 $, then the edge $ (i,j) $ is in $ G $ with probability $ p $; otherwise the edge $ (i,j) $ is in $ G $ with probability $ q $.
\end{definition}

In classic literature on majority dynamics, the underlying graph is fixed throughout the process, and in our framework, such static models do not reflect the evolution of block structures. For example, if an agent changes their opinion from $-$ to $+$, the marginal distribution of the influence from their neighbors (i.e., edge connections) does not change, which was sampled as if they had opinion $-$. Therefore, to maintain the community structure, our graph needs to evolve in time to be consistent with the updates of the opinions. The most natural idea is to resample the edges at each round.  Depending on the extent to which  the edges are resampled, we introduce the following two models. 

\begin{definition}\label{def:Model}
For a graph $ G=(V,E) $ whose vertices are labeled by $ [2n+\Delta] $ with the binary opinions $ W_i \in \sth{\pm 1} $, let $ \bfW := (W_1,\dots,W_{2n+\Delta}) $ denote the sequence of opinions of each vertex. Given an initialization with $ n+\Delta $ vertices with opinion $ + $ and $ n $ vertices with opinion $ - $ on the graph $ \SBM(n+\Delta,n,p,q) $, consider the following coupled dynamics $ \pth{G_t,\bfW_t} $ of the vertex opinions and the graph, with $G_0=G$ and $\bfW_0=\bfW$:
\begin{enumerate}[(i)]
    \item (\emph{Markovian model}). For each day $ t \in\N $, after the opinions $ \bfW_t $ are determined based on $ (G_{t-1},\bfW_{t-1}) $, we update the graph $ G_{t}=(V,E_t) $ by resampling the edges between all pairs $ (i,j) $ based on the law of $ \SBM(|\Vp{t}|,|\Vm{t}|,p,q) $. The opinions $ \bfW_{t+1} $ are computed via the majority rule \eqref{eq:Majority_Rule} on the updated graph $ (G_t,\bfW_t) $.
    \item (\emph{Non-Markovian model}). For each day $ t \in\N$, after the opinions $ \bfW_t $ are determined based on $ (G_{t-1},\bfW_{t-1}) $, we update the graph $ G_t = (V,E_t) $ in the following way. For any pair $ (i,j) $, if $ \WW{i}{t}=\WW{i}{t-1} $ and $ \WW{j}{t}=\WW{j}{t-1} $, then keep  the connectivity condition between these nodes, i.e.,
    $$ \indc{(i,j) \in E_t} = \indc{(i,j) \in E_{t-1}}. $$
    Otherwise, we resample the pair $ (i,j) $ based on their updated opinions using the law of $ \SBM(|\Vp{t}|,|\Vm{t}|,p,q) $, i.e.,
    \begin{equation*}
    \prob{(i,j) \in E_t}=
    \left\{
    \begin{aligned}
    & p,& & \mbox{if} \  \WW{i}{t} \WW{j}{t}=1 ;\\
    & q, & & \mbox{otherwise}.
    \end{aligned}
    \right.    
    \end{equation*}
    The opinions $ \bfW_{t+1} $ are computed via the majority rule \eqref{eq:Majority_Rule} on the updated graph $ (G_t,\bfW_t) $.
\end{enumerate}
\end{definition}

\begin{center}
\begin{minipage}{0.48\linewidth}
\includegraphics[width=\linewidth]{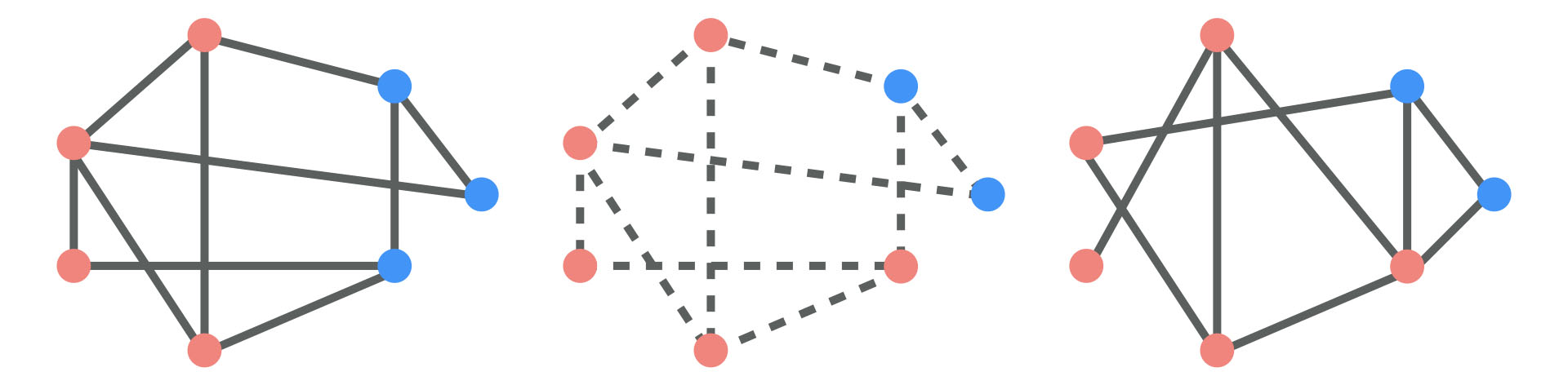}
\captionof{figure}{The Markovian model: all old connections are resampled}
\end{minipage}%
\hfill
\begin{minipage}{0.49\linewidth}
\includegraphics[width=\linewidth]{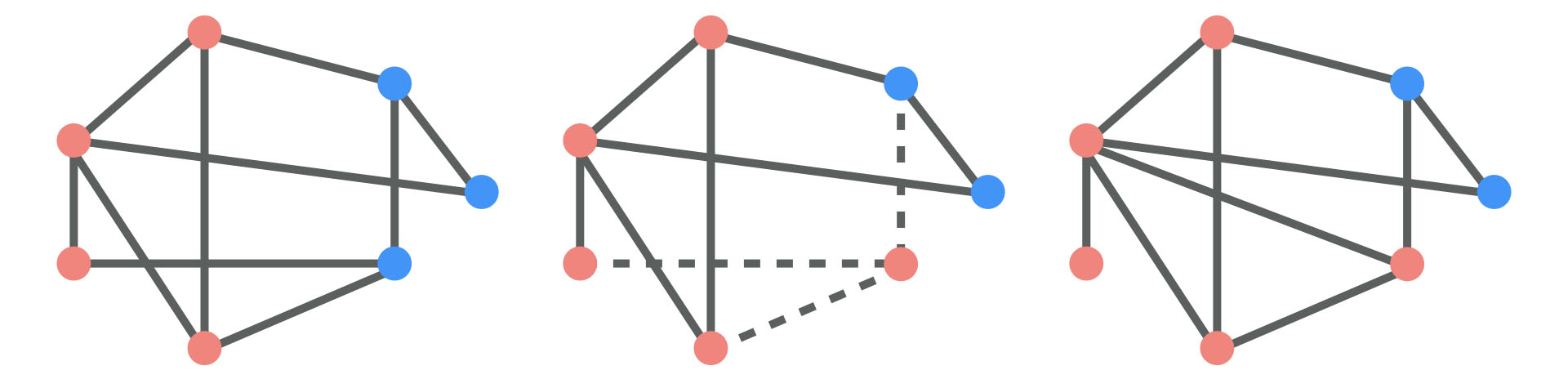}
\captionof{figure}{The non-Markovian model: only the dashed old connections are resampled}
\end{minipage}
\end{center}

\begin{remark}
In both models, our graph $G_t=(V,E_t)$ is a \emph{dynamic graph}, a graph whose topology evolves with time, in the sense of \cite{clementi2015distributed}. Dynamic stochastic block models were previously considered in the context of community detection by \cite{clementi2015distributed}, where the graph changes over time and the main task is to recover the evolving communities. More generally, Markovian networks have been investigated in the context of information spreading \cite{becchetti2011information} and random walks \cite{cai2020random}. In these works the time evolution of the graph is part of the environment on which a process runs. Our models are different in that the graph evolution is coupled to the opinion dynamics: the current opinions determine the edge-resampling probabilities, and in the non-Markovian model the memory of an edge depends on whether its endpoints changed opinions.
\end{remark}
 
\begin{remark}
For the Markovian model, it is easy to see that the underlying graph is always marginally a stochastic block model, with two blocks given by agents holding the two different opinions. Nevertheless, the marginal distribution of the non-Markovian model is in general \emph{not} simply a stochastic block model due to the dependency of the connections on the prior information of neighbors' opinions, making it harder to analyze.
\end{remark}

By definition, the evolution of the Markovian model on different days is independent, and hence the number of vertices with opinion $+$, $|V^{(t)}_+|$, $t\in\N_0$, forms a Markov chain on $\{0,\dots,2n+\Delta\}$ with absorbing states $\{0,2n+\Delta\}$. This explains the names of our models.

One motivation for our models is from social choice theory and opinion formation on social networks \cite{granovetter1978threshold,mossel2017opinion,nguyen2020dynamics}. In an election-like interpretation, the two opinions represent two parties. Same-party agents are more likely to communicate with one another, while cross-party communication is less frequent. The additional feature in our model is that the party labels are not fixed communities: they are the opinions produced by the dynamics. Thus, when an agent changes opinion, the local network around that agent should be rebuilt to reflect the updated opinion profile. In the non-Markovian model, once a pair of agents interact with each other, their social connection will not break unless one or both of them change opinions. The Markovian model, in contrast, can be regarded as the other extreme case, in which all agents interact with each other randomly at each round. It would be an interesting open question to study a model where the randomness of the resampling is intermediate.

\subsection{Main Results}

Let $ G=(V,E) \sim \SBM(n+\Delta,n,p,q) $, where $p>q$.  Since the case $ q=0 $ is trivial, we will always assume $p,q>0$. Indeed, if $ q=0 $, then $ G $ is disconnected and separated by the two disjoint communities, meaning that there will not be opinion changes.
Throughout the paper, for simplicity we fix two constants $(p,q)$ independent of $n$ with $0<q<p\leq 1$, but the initial bias  $\Delta=\Delta_n\in\N$ may depend on $ n $. We remark that our techniques generalize easily to the case $ (\log n)^{-c} \leq p,q \leq 1-(\log n)^{-c} $ with $p/q\geq 1+\delta$ for some absolute constant $ c>0 $ and any $\delta>0$, as well as the corresponding  cases with $q>p$. We now define some events that are of interest. 
\begin{definition}[Outcomes of the dynamics]
For the coupled majority dynamics $ (G_t,\bfW_t) $ defined in Definition \ref{def:Model}, consider the following events.
\begin{enumerate}[(i)]

\item The opinion $ + $ wins at day $ t $,
$$ \scrP_t := \sth{|\Vp{t}| = 2n+\Delta},\ t\in\N. $$
\item The opinion $ + $ wins eventually,
$$ \scrP := \sth{\lim_{t \to \infty} |\Vp{t}| =2n+\Delta }=\bigcup_{t\in\N} \sP_t. $$
\item The opinion $ - $ wins at day $ t $,
$$ \sM_t := \sth{ |\Vm{t}| = 2n+\Delta },\ t \in \N. $$
\item The opinion $ - $ wins eventually,
$$ \scrM := \sth{\lim_{t \to \infty}|\Vm{t}| = 2n+\Delta } = \bigcup_{t \in \N} \sM_t. $$
\item The dynamics never reaches consensus,
$$ \scrT := \pth{\scrP \cup \scrM}^c. $$
\end{enumerate}
\end{definition}

When we prove $\sT$ for the non-Markovian model below, the arguments establish the stronger conclusion that the process actually halts, in the sense that eventually no vertex changes opinion. The definition above is phrased as non-consensus so that it does not exclude, by definition, other non-consensus pathologies such as periodic behavior.

\begin{remark}\label{remark:heuristics}
Let us give some heuristics before presenting the main results. In the majority dynamics on SBM, the agents are more stubborn to be influenced by others with a different opinion than in the \ER model.  Suppose that we start from $|V^{(0)}_+|=n+\Delta$ and $|V^{(0)}_-|=n$, and we wish that the opinion $+$ wins eventually. Any agent with opinion $-$  will receive $\Bin(n+\Delta,q)$ many opinions $+$ and $\Bin(n-1,p)$ many opinions $-$. After taking expectation, this yields the natural guess that when $(n+\Delta)q\geq (n-1)p$ or (asymptotically) equivalently  $\Delta>(p-q)n/q$, the opinion $+$ will win, which is not difficult to confirm. On the other hand, having $\Delta\geq 1$  already breaks the symmetry between different opinions. So we expect for both models that the opinion $+$ dominates at some threshold between $\Delta=1$ and $\Delta=(p-q)n/q$, depending on the dependency of memories in the graph evolution.
\end{remark}

Our first result shows that the Markovian model exhibits the power-of-one behavior. Throughout the paper, constants denoted by $L=L(p,q)>0$ are positive and explicitly computable from $p,q$; the value of $L$ may change from one occurrence to the next unless a fixed value is specified.

\begin{theorem}\label{thmmodel1}
Consider the Markovian model on $ \SBM(n+\Delta,n,p,q) $, where $0<q<p\leq 1$.
\begin{enumerate}[(i)]
    \item\label{1i} Uniformly for $\Delta>0$ and $n\in\N$, there exists $L= L(p,q)>0 $ such that
    $$\P(\sP)\geq \frac{1}{2}+\frac{1}{L},$$
    i.e., the opinion $ + $ wins eventually with probability at least $ \tfrac{1}{2}+\tfrac{1}{L} $.
    \item\label{1ii} If $\Delta_n\to\infty$, then $\P(\sP)\to 1$, i.e., the opinion $ + $ wins asymptotically almost surely.
\end{enumerate}
\end{theorem}
\begin{remark}\label{rmk:time}
Let us note that for the Markovian model, the random walk evolves very slowly and the time till consensus will be exponentially increasing in $n$ for small $\Delta$ (e.g., for $\Delta$ with $(p-q)n/q-\Delta=\Omega(n)$, as can be seen from the proof of Theorem \ref{thmmodel1}). 
\end{remark}
Compared to the Markovian model, the non-Markovian model exhibits  a completely different behavior. Its update rule has a similar form to majority dynamics on \ER graphs studied in \cite{tran2020reaching,berkowitz2022central, sah2021majority}: vertices update synchronously by comparing the numbers of neighbors holding the two opinions, and the analysis starts from tail estimates for these competing neighborhood counts. The picture is still quite different for the following reasons.
\begin{enumerate}[(i)]
    \item The block structure with $p>q$ makes the connection densities distinct between and within each block, and thus we are comparing binomial distributions with different means (see Remark \ref{remark:heuristics}), and the order of such differences needs to be controlled.
    \item The dynamics halts for a large range of $\Delta$, due to the constraint $p>q$. For example, if at one round nobody changes opinion, the dynamics halts.
    \item For $\Delta\in\N$, with high probability there will not be any opinion change from $+$ to $-$ due to the significant difference of the binomial means (see Proposition \ref{prop:no+to-} below), which is not the case when $p=q$.
\end{enumerate}
Our main result can be stated as follows.
 
\begin{theorem}
Let $\Delta=\Delta_n\in\N$. Consider the non-Markovian model on $ \SBM(n+\Delta,n,p,q) $ where $0<q<p\leq 1$, and the constant \begin{align}H=H(p,q)=\frac{\sqrt{p(2-p-q)}}{q}.\label{H}\end{align}
\begin{enumerate}[(i)]
    \item For any $\ee>0$, if 
    \begin{align}\Delta_n\leq \left(\frac{p-q}{q}\right) n-H\sqrt{n\log n}-\ee \sqrt{\frac{n(\log\log n)^2}{\log n}},\label{eq:i}\end{align}
    then the dynamics halts asymptotically almost surely.\label{2i}

     \item For any sufficiently large constant $ L(p,q)>0 $, if  
     $$\Delta_n\geq \left(\frac{p-q}{q}\right)n+L\sqrt{n\log n},$$
     then $\P(\sP_1)\to 1$, i.e., the opinion $+$ wins on the first day asymptotically almost surely.\label{2ii}
     
    \item For any fixed constant $\delta>0$, if  $$\Delta_n\geq \left(\frac{p-q}{q}\right)n-(H-\delta)\sqrt{n\log n},$$
    then $\P(\sP_2)\to 1$, i.e., the opinion $+$ wins on the second day asymptotically almost surely.\label{2iii}
    
   \item For any $\ee>0$, if 
    \begin{align}\Delta_n\geq \left(\frac{p-q}{q}\right) n-H\sqrt{n\log n}+\ee \sqrt{\frac{n(\log\log n)^2}{\log n}},\label{eq:v}\end{align}
   then $\P(\sP_3)\to 1$, i.e., the opinion $+$ wins on the third day asymptotically almost surely.\label{2iv}
   \item Asymptotically almost surely, the opinion $ - $ will not win, i.e., $\P(\sM)=o(1)$.\label{2v}
\end{enumerate}
\label{thmmodel2}
\end{theorem}


Theorem \ref{thmmodel2} identifies the phase-transition threshold up to the
second-order leading term. The leading term $(p-q)n/q$ is the heuristic
threshold suggested by Remark \ref{remark:heuristics}, while the second-order
correction is $-H\sqrt{n\log n}$. More precisely, for every fixed
$\varepsilon>0$, if
$$
\Delta_n\leq \left(\frac{p-q}{q}\right) n-H\sqrt{n\log n}
-\varepsilon\frac{\sqrt n\log\log n}{\sqrt{\log n}},
$$
then the dynamics halts a.a.s.; if
$$
\Delta_n\geq \left(\frac{p-q}{q}\right) n-H\sqrt{n\log n}
+\varepsilon\frac{\sqrt n\log\log n}{\sqrt{\log n}},
$$
then the opinion $+$ wins by day three a.a.s. Thus, the transition between
halting and fast consensus is localized to a window of order
$\sqrt n\log\log n/\sqrt{\log n}$ around
$\frac{(p-q)n}{q}-H\sqrt{n\log n}$. The case $p=1$ is included in the same statement,
where $H=\sqrt{1-q}/q$.

The difficulty here is to analyze the behavior of the dynamics for multiple days (since the first few days do not give sufficient information on whether the dynamics will halt), especially when the underlying graph evolves with time. This contrasts with several static-graph majority dynamics models where one- or two-day estimates already determine the relevant coarse event with high probability \cite{benjamini2016convergence,fountoulakis2020resolution,zehmakan2020opinion,tran2020reaching,sah2021majority,berkowitz2022central}. In our non-Markovian model, the first day can be handled using the marginal SBM structure, but after some vertices change opinion the retained edges create memory, so the second and third days require separate conditional analyses. We also remark that the rates at which the probabilities converge to 1 in Theorem \ref{thmmodel2} can be analyzed explicitly in the proofs.



\begin{remark}
Among the literature on majority dynamics, the assignment of initial opinions can be either deterministic or random. Models involving random initial data have been studied in \cite{benjamini2016convergence,fountoulakis2020resolution}, among many others. More precisely, the set $V^{(0)}_+$ (and hence $V^{(0)}_-$) is now determined by a sequence of i.i.d.~random variables independent of everything else, where each vertex has probability $r\in(0,1)$ of holding opinion $+$ and $1-r$ of holding opinion $-$. In the framework of the Markovian model, if $r>1/2$ then $\sP$ holds asymptotically almost surely (in short, a.a.s.); if $r<1/2$ then $\sM$ holds a.a.s.; if $r=1/2$ then both $\sP$ and $\sM$ hold with probability tending to $1/2$. For the non-Markovian model, as a consequence of Theorem \ref{thmmodel2} and the central limit theorem, if $r\geq p/(p+q)$ then $\sP_2$ holds a.a.s.; if $r\leq q/(p+q)$ then $ \sM_2 $ happens a.a.s.; otherwise, $\sT$ holds a.a.s.~This completely characterizes the behavior of our models under random initial conditions.
\end{remark}

\subsection{Notation and  organization of the paper}
For $ n \in \N $,
let $ [n] $ denote the set of integers $ \{1,\dots,n\} $. We use small boldface letters to denote sequences. For a vector or a sequence $ \mathbf{a}=(a_1,\dots,a_n) $, we use $ |\mathbf{a}| = \sum_{i=1}^n |a_i| $ to denote the $ \ell_1 $ norm. The $j$-th component of $\mathbf{a}$ is also denoted by $\mathbf{a}(j)$.

For $ m,n \in \N $ and $ p \in [0,1] $, let $ \sfG(n,p) $ denote the \ER graph and
let $ \sfG(m,n,p) $ denote a random bipartite graph with $ m $ vertices on one side and $ n $ vertices on the other, each edge included independently with probability $ p $.

For $ \mu,\sigma \in \R $, let $ \calN(\mu,\sigma^2) $ be the normal distribution with mean $ \mu $ and variance $ \sigma^2 $. For $ n \in \N $ and $ p \in [0,1] $, we use $ \Bin(n,p) $ to denote the binomial distribution with parameters $ n $ and $ p $. When a binomial distribution appears inside a probability operator $ \P $, it should be interpreted as a binomial random variable with such a  distribution independent of everything else, unless, when a certain distribution appears multiple times,  the random variables are interpreted as being equal instead of being independent.

\smallskip

The remainder of the paper is organized as follows. In Section \ref{sec:Outline}, we briefly describe the proof strategies which will be different for the two distinct models. In Section \ref{sec:Pre}, we collect some auxiliary results, which include a summary of graph enumeration results of random graphs in Section \ref{sec:Enumeration} and nearly-optimal binomial tail bounds in Section \ref{sec:binomtails}. The complete proofs of the theorems are given in Section \ref{sec:proofs}. Finally, we provide some numerical simulations in Section \ref{sec:Num} and discuss some open questions in Section \ref{sec:Remark}.

\section{Proof Strategies}\label{sec:Outline}
\subsection{Markovian model}
In the Markovian model, thanks to the resampling of the whole graph, at each step the evolution of the model can be treated as the first round with initialization given by the updated opinions from the previous step. This implies that the number of vertices holding opinion $+$, $ \{|V^{(t)}_+|\}_{t\in\N_0}$, forms a Markovian random walk on $\{0,1,\dots,2n+\Delta\}$.  This reduces Theorem \ref{thmmodel1} to the analysis of a certain random walk on $\Z$. Such a random walk is symmetric around $n+\Delta/2$. Intuitively speaking, the walk is attracted by its two endpoints $0$ and $2n+\Delta$. Thus to determine $\P(\sP)$ we need to understand the behavior of the walk near the center $n+\Delta/2$. A crucial estimate, given by Lemma \ref{lemma:probratio}, states that near the center of the chain, the random walk rarely performs a move  of length greater than one, and that the ratio of probabilities of moving right by one to that of moving left by one is controlled from below by some constant $1+1/L$, uniformly in $n$ and $\Delta$. Roughly speaking, this stems from the fact that the binomial tails are exponentially decreasing away from the mean. Theorem \ref{thmmodel1} then follows from Gambler's Ruin estimates, together with the observation that being close to an endpoint ensures the chain to reach the endpoint with high probability, which is given by Proposition \ref{prop:model2winseventually}.

\subsection{Non-Markovian model}
For the non-Markovian model, due to the dependency on the memory of previous steps, we need to track the evolution of graphs more carefully. The first step is to understand  the joint distribution of the degrees of the vertices on the initial day. The marginal distribution of the degree of a fixed vertex is binomial, whereas the degrees are not independent for different neighbors. We bypass this difficulty by replacing  the degrees with a simpler probabilistic model, in which the degrees of the vertices form a sequence of  independent binomial random variables. This idea is based on graph enumeration results of random graphs by McKay and Wormald \cite{mckay1997degree}, which roughly states that the distributions of the degrees in a random graph are approximately conditionally independent. 
We address the reduction to the independent degree model in Section \ref{sec:Enumeration}.

We want to emphasize that although the previous work \cite{sah2021majority} also uses graph enumeration techniques, our method is very different. In our work, graph enumerations are only used in the analysis for the initial day. Moreover, graph enumerations cannot be applied for the follow-up days in the process. This is because the dependency of past memories makes the edge connections highly correlated. The edge connections are hard to track along the process, and thus it is impossible to apply similar analysis for steps after day $ 1 $ as in the static \ER model.

This also clarifies the relation with the coarse-event union-bound approach in \cite{tran2020reaching}. In the static \ER setting, once the graph is sampled, one can control the relevant first- or second-day events by applying binomial tail estimates and union bounds over the vertices. We use the same broad philosophy for the first day, but after day $1$ the non-Markovian rewiring rule introduces conditional dependencies through the retained edges. The second- and third-day arguments therefore have to condition on the realized opinion changes and track only the new randomness generated by vertices that changed opinion.

The proof of the halting side of Theorem \ref{thmmodel2} is based on the fact
that, if no vertex changes opinion in one round, then the non-Markovian
dynamics becomes stationary. If
$$
\Delta_n\leq \left(\frac{p-q}{q}\right) n-H\sqrt{n\log n}
-\varepsilon\frac{\sqrt n\log\log n}{\sqrt{\log n}},
$$
the number of initially negative vertices that change to $+$ on the first
day is small. The main task is then to show that these few changed vertices do
not create a large enough reservoir of new $+$-neighbors to restart the
cascade. This is done by combining local moderate-deviation estimates for the
initial binomial comparison with Proposition \ref{prop:no+to-}  below, which states that there are no vertices changing opinion from $ + $ to $ - $ a.a.s. 

On the other hand, if
$$
\Delta_n\geq \left(\frac{p-q}{q}\right) n-H\sqrt{n\log n}
+\varepsilon\frac{\sqrt n\log\log n}{\sqrt{\log n}},
$$
there are enough first-day changes from $-$ to $+$ to form a reservoir.
The newly resampled edges from this reservoir give many remaining negative
vertices additional $+$-neighbors. This pushes the process into the easier
supercritical regime, where Proposition \ref{prop2lastdaywins} implies that
the opinion $+$ wins by day three. The one- and two-day sufficient conditions
in Theorem \ref{thmmodel2} follow from first-step estimates of the same
binomial comparisons, using graph enumeration techniques mentioned above.

\section{Preliminaries}\label{sec:Pre}

\subsection{Reduction to the independent model}\label{sec:Enumeration}

The swapping of opinions in the majority dynamics is based on the degrees of the vertices in the stochastic block model, but the distribution of the true degrees  of an SBM is hard to analyze due to the constraint from the graph structure. To overcome this issue, as mentioned in the outline of proofs, we rely on the graph enumeration technique developed by McKay and Wormald \cite{mckay1990asymptotic,mckay1997degree}. At a high level, their work implies that the degrees of random graphs look conditionally independent. 
These techniques apply to the non-Markovian model in the proof of Theorem \ref{thmmodel2}.


To make our paper self-contained, in this section we review some probabilistic models for the degree sequences, which will be used to give a quantitative version of the aforementioned high-level idea of graph enumerations. To begin with, let us define the domains where the degree sequences are defined, following \cite{mckay1997degree}.

\begin{definition}[Degree sequence domains]
Denote $ I_n=\{0,\dots,n-1\}^n $. Let $ E_n $ be the even sum sequences in $ I_n $. The elements of these sets are typically denoted by the small bold letter $ \bd $. For the bipartite setting, similarly denote $ I_{m,n}=\{0,\dots,n\}^m \times \{0,\dots,m\}^n $. Let $ E_{m,n} $ be the sequences with equal sums on both sides. The elements of these sets are typically denoted by small boldface letters $ \bs $ of length $ m $ and $ \bt $ of length $ n $. The corresponding random variables will be denoted by capital boldface letters.
\end{definition}

The distribution of the true degree sequence in $ \sfG(n,p) $ and $ \sfG(m,n,p) $ is denoted as follows.

\begin{definition}[True degree models]
Let $ \calD_p^n $ be the degree sequence distribution of the \ER graph $ \sfG(n,p) $, which is a random variable supported on $ E_{n} \subset I_{n} $. For the bipartite setting, let $ \calD_{p}^{m,n} $ be the degree sequence distribution of the random bipartite graph $ \sfG(m,n,p) $, which is a random variable supported on $ E_{m,n} \subset I_{m,n} $.
\end{definition}

For approximations of the true degree model, we introduce the following models.

\begin{definition}[Independent degree models]
Let $ \calB_p^n $ be the distribution of $ n $ independent copies of $ \Bin(n-1,p) $ random variables, which is supported on $ I_n $. For the bipartite setting, let $ \calB_p^{m,n} $ be the distribution of $ m $ independent $ \Bin(n,p) $ variables and $ n $ independent $ \Bin(m,p) $ variables, which is supported on $ I_{m,n} $. 
\end{definition}

\begin{definition}[Conditioned degree models]
Let $ \calE_p^n $ be the distribution of $ \calB_p^n $ conditioned on having even sum, which is supported on $ E_n $. 
\end{definition}

\begin{definition}[Integrated degree models]
Let $ \calI_p^n $ be the distribution defined as follows: First sample $ p' \sim \calN(p,\tfrac{p(1-p)}{n(n-1)}) $, conditional on $ p' \in (0,1) $; then sample from $ \calE_{p'}^{n} $. 
\end{definition}

Using the models introduced above, we now state the following necessary preliminary result.

\begin{theorem}[{\cite[Theorem 3]{mckay1990asymptotic} and \cite[Theorem 3.6]{mckay1997degree}}]\label{thm:Enumeration_ER}
There exists $ c>0 $ so that the following is true. Let $ n \geq 2 $ and suppose that $ (\log n)^{-1/4} \leq p \leq 1-(\log n)^{-1/4} $. There is an event $ B_p^n \subset I_n $ such that $ \P_{\calD_p^n}(B_p^n) = n^{-\omega(1)} $
and uniformly for all $ \bd \in I_n \backslash B_p^n $ we have
$$ \P_{\calD_p^n} \pth{\bD=\bd} = \pth{1+O(n^{-c})} \P_{\calI_p^n} \pth{\bD=\bd}. $$
\end{theorem}

In the remainder of this paper, when dealing with probabilities concerning random graphs, the notation $ \P $ (without mentioning the probability model explicitly) represents the probability measure of the true degree models. For example, the conclusion of Theorem \ref{thm:Enumeration_ER} can be written as $$ \P \pth{\bD=\bd} = \pth{1+O(n^{-c})} \P_{\calI_p^n} \pth{\bD=\bd}. $$

Note that the degree sequence in $ \SBM(m,n,p,q) $ is sampled from three independent subgraphs: $ \sfG(m,p) $, $ \sfG(n,p) $ and $ \sfG(m,n,q) $. Using Theorem \ref{thm:Enumeration_ER}, when considering the marginal probabilities on a single block, we can replace the true degrees of the \ER subgraph with the integrated degree model, and henceforth reduce it to the independent degree model. Specifically, we have the following result. In the displays below, the event involving $|\Vp{0}\cap\Vp{1}|$ uses only the one-side marginal degree sequence of the initial $+$ block together with the cross-degrees, and the event involving $|\Vm{0}\cap\Vm{1}|$ uses only the corresponding one-side marginal of the initial $-$ block together with the cross-degrees. We are not replacing the full joint SBM degree sequence by independent marginals.

\begin{lemma}\label{lem:True_to_Conditioned2}
Assume $0<q<p<1$. For any sufficiently large fixed $ L>0$, we have
\begin{multline*}
 \P_{\calD_{p}^{n+\Delta},\calB_{q}^{n+\Delta,n}}\pth{|\Vp{0} \cap \Vp{1}|=x}\\
 = \pth{1+O(n^{-c})} \int_{R} \P_{ \calB_{r_1}^{n+\Delta}, \calB_{q}^{n+\Delta,n} } \pth{ |\Vp{0} \cap \Vp{1}|=x } \dd \mu_{1}(r_1) + O(n^{-\omega(1)}),
 \end{multline*}
and
\begin{multline*}
 \P_{\calD_{p}^{n},\calB_{q}^{n+\Delta,n}}\pth{|\Vm{0} \cap \Vm{1}|=y}\\
 = \pth{1+O(n^{-c})} \int_{R} \P_{ \calB_{r_2}^{n}, \calB_{q}^{n+\Delta,n} } \pth{ |\Vm{0} \cap \Vm{1}|=y } \dd \mu_{2}(r_2) + O(n^{-\omega(1)}).
 \end{multline*} 
where the measure $ \mu_1 $ is the distribution $ \calN \pth{p,\tfrac{p(1-p)}{(n+\Delta)(n+\Delta-1)}} $, $ \mu_2 $ is the distribution $ \calN \pth{p,\tfrac{p(1-p)}{n(n-1)}} $, and the integral domain is given by
\begin{equation}
R := \left\{ r \in [0,1] :~ |r-p| \leq  \frac{L\log n}{n} \right\}.
\label{eq:R}
\end{equation}
\end{lemma}

\begin{proof}
For a given graph, the swapping of opinions purely depends on the degree information. This implies that the sizes of the swapped vertices, as random variables, are measurable with respect to the degree sequences. Note that the degree sequences are sampled from three independent random graphs: two \ER graphs $ \sfG(n+\Delta,p) $, $ \sfG(n,p) $, and a random bipartite graph $ \sfG(n+\Delta,n,q) $, i.e., the randomness of this event comes from the true degree models $ \calD_{p}^{n+\Delta} $, $ \calD_{p}^{n} $ and $ \calD_{q}^{n+\Delta,n} $ independently. Since we are interested in the marginal behavior of a single block, by Theorem \ref{thm:Enumeration_ER}, we can replace the true degree models $ \calD_{p}^{n+\Delta} $ and $ \calD_{p}^{n} $ with the integrated degree models $ \calI_{p}^{n+\Delta} $ and $ \calI_{p}^{n} $ up to a $ 1+O(n^{-c}) $ multiplicative factor and an additive error term of size $ O(n^{-\omega(1)}) $.

For the two \ER subgraphs, let $ \bd_1 \in I_{n+\Delta} $ be the degree sequence of length $ n+\Delta $ and $ \bd_2 \in I_n $ be the degree sequence of length $ n $. For the bipartite subgraph, we use $ (\bs,\bt) \in I_{n+\Delta,n} $ to denote the degree sequences of length $ n+\Delta $ and $ n $, respectively. We treat the swapped sets $ \Vp{1} $ and $ \Vm{1} $ as functions of $ (\bd_1,\bd_2,(\bs,\bt)) $, and extend these functions in the obvious way if the total sum in $ I_{n+\Delta} $ or $ I_{n} $ is not even, or the sums on both sides of $ I_{n+\Delta,n} $ do not match.

Using these notation, by Theorem \ref{thm:Enumeration_ER}, we have
\begin{multline}\label{eq:True_to_Integrated_1}
\P_{\calD_{p}^{n+\Delta},\calB_{q}^{n+\Delta,n}} \pth{|\Vp{0} \cap \Vp{1}|=x}\\
= \pth{1+O(n^{-c})} \P_{\calI_{p}^{n+\Delta},\calB_{q}^{n+\Delta,n}} \pth{ |\Vp{0} \cap \Vp{1}|=x } + O(n^{-\omega(1)})
\end{multline}
and
\begin{equation}\label{eq:True_to_Integrated_2}
\P_{\calD_{p}^{n},\calB_{q}^{n+\Delta,n}} \pth{|\Vm{0} \cap \Vm{1}|=y}
= \pth{1+O(n^{-c})} \P_{\calI_{p}^{n},\calB_{q}^{n+\Delta,n}} \pth{ |\Vm{0} \cap \Vm{1}|=y } + O(n^{-\omega(1)}).
\end{equation}
By the definition of integrated degree models, we further have
\begin{equation*}
\P_{\calI_{p}^{n+\Delta},\calB_{q}^{n+\Delta,n}} \pth{ |\Vp{0} \cap \Vp{1}|=x }
= \frac{1}{\int_{[0,1]} \dd \mu_1(r_1)} \int_{[0,1]} \P_{\calE_{r_1}^{n+\Delta},\calB_{q}^{n+\Delta,n}} \pth{ |\Vp{0} \cap \Vp{1}|=x } \dd \mu_1(r_1)
\end{equation*}
and
\begin{equation*}
\P_{\calI_{p}^{n},\calB_{q}^{n+\Delta,n}} \pth{ |\Vm{0} \cap \Vm{1}|=y }
= \frac{1}{\int_{[0,1]} \dd \mu_2(r_2)} \int_{[0,1]} \P_{\calE_{r_2}^{n},\calB_{q}^{n+\Delta,n}} \pth{ |\Vm{0} \cap \Vm{1}|=y } \dd \mu_2(r_2).
\end{equation*}
Consider the random variables
$$ Z_1 \sim \calN \pth{p, \frac{p(1-p)}{(n+\Delta)(n+\Delta-1)}},\ \  Z_2 \sim \calN \pth{p,\frac{p(1-p)}{n(n-1)}}. $$
The Gaussian tail bound gives that for $ i=1,2 $,
$$ \prob{|Z_i-p| \leq \frac{L \log n}{n}  } = 1-O(n^{-\omega(1)}). $$
This yields
\begin{equation*}
\P_{\calI_{p}^{n+\Delta},\calB_{q}^{n+\Delta,n}} \pth{ |\Vp{0} \cap \Vp{1}|=x }
= \int_{R} \P_{\calE_{r_1}^{n+\Delta},\calB_{q}^{n+\Delta,n}} \pth{ |\Vp{0} \cap \Vp{1}|=x } \dd \mu_1(r_1) + O(n^{-\omega(1)}),
\end{equation*}
and
\begin{equation*}
\P_{\calI_{p}^{n},\calB_{q}^{n+\Delta,n}} \pth{ |\Vm{0} \cap \Vm{1}|=y } = \int_{R} \P_{\calE_{r_2}^{n},\calB_{q}^{n+\Delta,n}} \pth{ |\Vm{0} \cap \Vm{1}|=y } \dd \mu_2(r_2) + O(n^{-\omega(1)}),
\end{equation*}where $R$ is given by \eqref{eq:R}.

Finally, we remove the evenness constraints in the conditioned models $ \calE_{r_1}^{n+\Delta} $ and $ \calE_{r_2}^{n} $. 
Note that the Bayes' rule implies
\begin{equation*}
\P_{\calE_{r_1}^{n+\Delta},\calB_{q}^{n+\Delta,n}} \pth{ |\Vp{0} \cap \Vp{1}|=x } = \frac{\P_{\calB_{r_1}^{n+\Delta},\calB_{q}^{n+\Delta,n}} \pth{ |\Vp{0} \cap \Vp{1}|=x, |\bD_+| \in 2 \Z }}{\P_{\calB_{r_1}^{n+\Delta}} \pth{|\bD_+| \in 2 \Z} },
\end{equation*}
and
\begin{equation*}
\P_{\calE_{r_2}^{n},\calB_{q}^{n+\Delta,n}} \pth{ |\Vm{0} \cap \Vm{1}|=y } = \frac{\P_{\calB_{r_2}^{n},\calB_{q}^{n+\Delta,n}} \pth{ |\Vm{0} \cap \Vm{1}|=y , |\bD_-| \in 2 \Z }}{ \P_{\calB_{r_2}^{n}} \pth{|\bD_-| \in 2 \Z} },
\end{equation*}
where $ \bD_+ $ and $ \bD_- $ are the degree sequences of $ \Vp{0} $ and $ \Vm{0} $, respectively. Using the arguments in \cite[Equation (2.4)]{sah2021majority}, we have
\begin{align*}
& \frac{\P_{\calB_{r_1}^{n+\Delta},\calB_{q}^{n+\Delta,n}} \pth{ |\Vp{0} \cap \Vp{1}|=x, |\bD_+| \in 2 \Z }}{\P_{\calB_{r_1}^{n+\Delta}} \pth{|\bD_+| \in 2 \Z} }\\
&= \frac{ (\frac{1}{2} + O(\exp(-n))) \P_{\calB_{r_1}^{n+\Delta},\calB_{q}^{n+\Delta,n}} \pth{ |\Vp{0} \cap \Vp{1}|=x} + O\pth{\exp(-\Omega(n))} }{\frac{1}{2} + O(\exp(-n))}\\
&= \P_{\calB_{r_1}^{n+\Delta},\calB_{q}^{n+\Delta,n}} \pth{ |\Vp{0} \cap \Vp{1}|=x} + O\pth{\exp(-\Omega(n))},
\end{align*}
and similarly
$$ \frac{\P_{\calB_{r_2}^{n},\calB_{q}^{n+\Delta,n}} \pth{ |\Vm{0} \cap \Vm{1}|=y , |\bD_-| \in 2 \Z }}{ \P_{\calB_{r_2}^{n}} \pth{|\bD_-| \in 2 \Z} } = \P_{\calB_{r_2}^{n},\calB_{q}^{n+\Delta,n}} \pth{ |\Vm{0} \cap \Vm{1}|=y} + O\pth{\exp(-\Omega(n))}. $$
Plugging back into  \eqref{eq:True_to_Integrated_1} and \eqref{eq:True_to_Integrated_2} completes the proof.
\end{proof}

\subsection{Tail estimates of binomial distributions}\label{sec:binomtails}
For the proofs of the main theorems, we will frequently use the marginal probabilities that the opinion of a certain vertex flips at the first step in majority dynamics on $ \SBM(n+\Delta,n,p,q) $, defined by
\begin{align}p_{-+}:=\prob{\Bin(n+\Delta,q)>\Bin(n-1,p) },\ p_{+-}:=\prob{\Bin(n,q)>\Bin(n+\Delta-1,p)}.\label{eq:prbpbr}\end{align}

In this section, we collect some technical results concerning the tail behavior of binomial random variables and give upper and lower bounds for quantities of the form  \eqref{eq:prbpbr}. Throughout, we consider $0<q<p\leq 1$, and $\tilde{p}\in[0,1]\cap[p-L(\log n)/n,p+L(\log n)/n]$. Define $\tilde{p}_{-+}:=\prob{\Bin(n+\Delta,q)>\Bin(n-1,\tilde{p}) }$ similarly as in \eqref{eq:prbpbr}. Recall  Hoeffding's inequality that for $k\leq n\tilde{p}$,
\begin{align}
    \P(\Bin(n,\tilde{p})\leq k)\leq L\exp \pth{-2n \pth{\tilde{p}-\frac{k}{n}}^2}.\label{eq:Hoeffding}
\end{align}

The following Lemma gives upper and lower bounds for the probability in \eqref{eq:Hoeffding} in the case $np-k=O(\sqrt{n\log n})$. For simplicity of notation, let us define

\begin{equation}\label{eq:Delta'}
\Delta'=\Delta'_n := \pth{\frac{p-q}{q}} n-\Delta_n
\end{equation}

\begin{lemma}\label{lemmaBinomialtail1}
Assume $0<q<p\leq 1$. For every fixed $A>0$, there exists $ M=M(A,p,q)>0 $ such that if $ 0\leq \Delta'_n\leq A\sqrt{n\log n} $,
then
\begin{align}
\frac{1}{M\sqrt{\log{n}}}\exp\pth{- \frac{C\D^2(n+\Delta)}{n^2}} \leq \P(\Bin(n+\Delta,q)>n\tilde{p})\leq M\exp \pth{- \frac{C\D^2(n+\Delta)}{n^2}},\label{eqpbrbound}
\end{align}
where $C=C(p,q)=q^3/(2(1-q)p^2)$.
\end{lemma}

\begin{proof}Assume first $0<q<p<1$. Recall from    \cite[Lemma 4.7.2]{ash2012information} the classical tail bounds for binomial distribution (here $p$ may depend on $n$):
\begin{align}
\frac{1}{L\sqrt{n}}\exp \pth{ -(n+\Delta)\DD{\frac{np}{n+\Delta}}{q} }&\leq \P(\Bin(n+\Delta,q)=np)\leq\P(\Bin(n+\Delta,q)\geq np)\nonumber\\
&\leq L\exp \pth{ -(n+\Delta)\DD{\frac{np}{n+\Delta}}{q} },\label{eqBinomialtail}
\end{align}
where the Kullback-Leibler divergence $ \DD{\cdot}{\cdot} $ is given by
$$\DD{a}{p}:=a\log\left(\frac{a}{p}\right)+(1-a)\log\left(\frac{1-a}{1-p}\right).$$
Here and later, for simplicity we may assume $np$ as well as any other quantities that are $\omega(1)$ to be integers. Applying the floor or ceiling functions will not change the final results. 

Using the inequalities $x-x^2/2\leq\log(1+x)\leq x-x^2/2+Lx^3$ and $-x-x^2/2-Lx^3\leq\log(1-x)\leq -x-x^2/2$ for $0<x<1$, we compute
\begin{align*}
    &\hspace{0.5cm}\DD{\frac{n\tilde{p}}{n+\Delta}}{q}\\
    &=\frac{np}{\frac{p}{q}n-\D}\log \left(\frac{\frac{np}{q}}{\frac{p}{q}n-\D}\right)+\left(1-\frac{np}{\frac{p}{q}n-\D}\right)\log\left(\frac{1-\frac{np}{\frac{p}{q}n-\D}}{1-q}\right)+o\left(\frac{1}{n}\right)\\
    &\geq \frac{np}{\frac{p}{q}n-\D}\left(\frac{\D}{\frac{p}{q}n-\D}-\frac{\D^2}{2(\frac{p}{q}n-\D)^2}\right)\\
    &\hspace{0.2cm}+\left(1-\frac{np}{\frac{p}{q}n-\D}\right)\left(-\frac{q\D}{(1-q)(\frac{p}{q}n-\D)}-\frac{(q\D)^2}{2(1-q)^2(\frac{p}{q}n-\D)^2}-\frac{L(q\D)^3}{(1-q)^3(\frac{p}{q}n-\D)^3}\right)+o\left(\frac{1}{n}\right)\\
    &=\left(\frac{q^3}{2(1-q)p^2}\right)\frac{\D^2}{n^2}+o\left(\frac{1}{n}\right),
\end{align*}
where we used in the last step that $\D(n)=o(n^{2/3})$. Similarly,
\begin{align*}
   \DD{\frac{n\tilde{p}}{n+\Delta}}{q}  &\leq \frac{np}{\frac{p}{q}n-\D}\left(\frac{\D}{\frac{p}{q}n-\D}-\frac{\D^2}{2(\frac{p}{q}n-\D)^2}+\frac{L\D^3}{(\frac{p}{q}n-\D)^3}\right)\\
    &\hspace{1cm}+\left(1-\frac{np}{\frac{p}{q}n-\D}\right)\left(-\frac{q\D}{(1-q)(\frac{p}{q}n-\D)}-\frac{(q\D)^2}{2(1-q)^2(\frac{p}{q}n-\D)^2}\right)+o\left(\frac{1}{n}\right)\\
    &=\left(\frac{q^3}{2(1-q)p^2}\right)\frac{\D^2}{n^2}+o\left(\frac{1}{n}\right).
\end{align*}
Plugging these estimates into \eqref{eqBinomialtail} gives the desired upper bound
\[\P(\Bin(n+\Delta,q)\geq n\tilde{p})\leq L\exp\left(-\frac{C\D^2(n+\Delta)}{n^2}\right)\]
and 
\begin{align}
    \P(\Bin(n+\Delta,q)=n\tilde{p})\geq \frac{1}{L\sqrt{n}}\exp\left(-\frac{C\D^2(n+\Delta)}{n^2}\right).\label{eqequaltonp}
\end{align}
Let us refine the lower bound \eqref{eqequaltonp} to get a lower bound for $\P(\Bin(n+\Delta,q)\geq n\tilde{p})$. Note that
\begin{align}
    \frac{\P\left(\Bin(n+\Delta,q)=n\tilde{p}+\sqrt{\frac{n}{\log n}}\right)}{\P(\Bin(n+\Delta,q)=n\tilde{p})}&=\left(\frac{q}{1-q}\right)^{\sqrt{\frac{n}{\log n}}} \, \frac{\binom{n+\Delta}{n\tilde{p}+\sqrt{\frac{n}{\log n}}}}{ \binom{n+\Delta}{n\tilde{p}}}\nonumber\\
    &\geq \left(\frac{q(n+\Delta-n\tilde{p}-\sqrt{\frac{n}{\log n}})}{(1-q)n\tilde{p}}\right)^{\sqrt{\frac{n}{\log n}}}\nonumber\\
    &\geq \left(1-M\sqrt{\frac{\log n}{ n}}\right)^{\sqrt{\frac{n}{\log n}}}\geq \frac{1}{M},\label{eqratio}
\end{align}where the constant $M$ may not be the same on each occurrence.
By unimodality of the probability mass function of the binomial distribution and \eqref{eqequaltonp}, 
\begin{align*}
    \P(\Bin(n+\Delta,q)\geq n\tilde{p})&\geq \P\left(\Bin(n+\Delta,q)= n\tilde{p}+\sqrt{\frac{n}{\log n}}\right)\sqrt{\frac{n}{\log n}}\\
    &\geq \frac{1}{M\sqrt{n}}\exp\left(-\frac{C\D^2(n+\Delta)}{n^2}\right)\sqrt{\frac{n}{\log n}}\\
    &\geq \frac{1}{M\sqrt{\log n}}\exp\left(-\frac{C\D^2(n+\Delta)}{n^2}\right).
\end{align*}
This completes the proof for $0<q<p<1$.

When $p=1$, $Y_n\sim\Bin(n-1,1)$ is deterministic, so
$S_n=Y_n-X_n=n-1-X_n$, with
$X_n\sim\Bin(n+\Delta_n,q)$. Thus, $\mathbb P(X_n>n)$ differs
from $\mathbb P(S_n<0)$ only by an $O(1)$ threshold shift, covered by \eqref{eq:tail} of Lemma \ref{lem:local}, which gives the exponent
\[
\frac{q^2(\Delta'_n)^2}{2(1-q)n}.
\]
Since $n+\Delta_n=n/q-\Delta'_n$ and $\Delta'_n=O(\sqrt{n\log n})$,
\[
\frac{q^2(\Delta'_n)^2}{2(1-q)n}
=
\frac{q^3(\Delta'_n)^2(n+\Delta_n)}{2(1-q)n^2}+o(1),
\]
which is the exponent in Lemma 3.2 for $p=1$.
\end{proof}

We now apply Lemma \ref{lemmaBinomialtail1} to obtain estimates on the (fundamentally important) probability that one binomial random variable  is larger than the other.

\begin{lemma}\label{lemmaBinomialtail2}
Assume $0<q<p<1$. For every fixed $A>0$, there exists $ M=M(A,p,q)>0 $ such that if $0\leq \Delta'\leq A\sqrt{n\log n}$, then
$$\frac{1}{M (\log n)} \exp \pth{-\frac{C'\Delta'^2}{n}} \leq \tilde{p}_{-+}\leq M(\log n) \exp \pth{-\frac{C'\Delta'^2}{n}},$$
where $C'=C'(p,q)$ is given by
\begin{align}C'(p,q)=\frac{q^2}{2p(2-p-q)}.\label{eq:cpq}\end{align}
\end{lemma}

\begin{proof}
Denote by $K=K(n,\tilde{p},q,\Delta)\in\Z\cap[(n+\Delta)q,n\tilde{p}]$ that solves the minimization problem
\begin{align*}\min_{K\in\Z\cap[(n+\Delta)q,n\tilde{p}]}\left(\frac{1}{2(1-q)}\frac{q(n+\Delta)(K-(n+\Delta)q)^2}{K^2}+\frac{1}{2\tilde{p}}\frac{(1-\tilde{p})n(n-K-n(1-\tilde{p}))^2}{(n-K)^2}\right)\end{align*} and $C(n,\tilde{p},q,\Delta)$ the attained minimum.
One checks using $0\leq \Delta'_n=O(\sqrt{n\log n})$ that \begin{align*}C(n,\tilde{p},q,\Delta)&=\min_{K\in[(n+\Delta)q,np]}\left(\frac{np(K-(n+\Delta)q)^2}{2(1-q)(np)^2}+\frac{(1-p)n(n-K-n(1-p))^2}{2p(n(1-p))^2}\right)+O(1)\\
&=C'(p,q)\frac{\Delta'^2}{n}+O(1)\end{align*} where $C'(p,q)$ is given by \eqref{eq:cpq}, and this holds uniformly for $\tilde{p}\in[0,1]\cap[p-L(\log n)/n,p+L(\log n)/n]$. We have by using independence and \eqref{eqpbrbound} that
\begin{align*}
    \tilde{p}_{-+}&\geq \P(\Bin(n+\Delta,q)\geq K)~\P(\Bin(n-1,\tilde{p})\leq K)\\
    &\geq \frac{1}{M\log n}\exp\left(-\frac{1}{2(1-q)}\frac{q(n+\Delta)(K-(n+\Delta)q)^2}{K^2}-\frac{1}{2\tilde{p}}\frac{(1-\tilde{p})n(n-K-n(1-\tilde{p}))^2}{(n-K)^2}\right)\\
    &\geq  \frac{1}{M\log n}\exp\left(-\frac{C'\Delta'^2}{n}\right).
\end{align*}
This gives the lower bound as desired.

For the upper bound, since $0\leq \Delta'=O(\sqrt{n\log n})$, we may chop the interval $[(n+\Delta)q,(n-1)\tilde{p}]$ into at most $M\log n$ pieces of lengths $\sqrt{n/\log n}$. This gives
\begin{align*}
    \tilde{p}_{-+}&\leq \sum_{j=1}^{M\log n}\P\left(\Bin(n-1,\tilde{p})<j\sqrt{\frac{n}{\log n}}+(n+\Delta)q\right)\P\left(\Bin(n+\Delta,q)>(j-1)\sqrt{\frac{n}{\log n}}+(n+\Delta)q\right)\\
    &\hspace{2cm}+\P(\Bin(n-1,\tilde{p})<(n+\Delta)q)+\P(\Bin(n+\Delta,q)>(n-1)\tilde{p}).
\end{align*}
By the same arguments as in \eqref{eqratio}, the first term is bounded by
\begin{align*}
    &\hspace{0.5cm}\sum_{j=1}^{M\log n}\P\left(\Bin(n-1,\tilde{p})<j\sqrt{\frac{n}{\log n}}+(n+\Delta)q\right)\P\left(\Bin(n+\Delta,q)>(j-1)\sqrt{\frac{n}{\log n}}+(n+\Delta)q\right)\\
    &\leq M\sum_{j=1}^{M\log n}\P\left(\Bin(n-1,\tilde{p})<j\sqrt{\frac{n}{\log n}}+(n+\Delta)q\right)\P\left(\Bin(n+\Delta,q)>j\sqrt{\frac{n}{\log n}}+(n+\Delta)q\right)\\
    &\leq M\sum_{j=1}^{M\log n}\exp(-C(n,\tilde{p},q,\Delta))\\
    &\leq M(\log n)\exp\left(-\frac{C'\Delta'^2}{n}\right),
\end{align*}where the constant $M$ may not be the same on each occurrence.
On the other hand, it is easy to check using Lemma \ref{lemmaBinomialtail1} that the remaining two terms satisfy
$$\P(\Bin(n-1,\tilde{p})<(n+\Delta)q)+\P(\Bin(n+\Delta,q)>(n-1)\tilde{p})\leq M(\log n)\exp\left(-\frac{C'\Delta'^2}{n}\right).$$This finishes the proof of the upper bound.
\end{proof}

We need one more result on moderate deviations of the difference of two binomial random variables. Set 
\begin{align}
\tau^2 = p(2-p-q),\qquad
C' = \frac{q^2}{2\tau^2} = \frac{1}{2H^2},\qquad
I_n = \frac{C'(\Delta'_n)^2}{n}.\label{eq:defs}
\end{align}

\begin{lemma}\label{lem:local}
Fix constants $0<a_0<A_0<\infty$ and assume
\begin{equation}\label{eq:md-range}
    a_0\sqrt{n\log n}\le\Delta'_n\le A_0\sqrt{n\log n}.
\end{equation}
Let $S_n=Y_n-X_n$ with $Y_n\sim\Bin(n-1,p)$ and $X_n\sim\Bin(n+\Delta_n,q)$ independent.  There is $\lambda_n=q\Delta'_n(\tau^2 n)^{-1}(1+o(1))\asymp\sqrt{\log n/n}$ such that:
\begin{enumerate}[(i)]
    \item uniformly for integer $k=o(\sqrt n)$,
    \begin{equation}\label{eq:point}
        \P(S_n=k)
        =
        \frac{\exp\{-I_n+o(1)\}\exp\{\lambda_n k\}}
             {\sqrt{2\pi\tau^2 n}};
    \end{equation}
    \item for constants $0<c<C<\infty$ depending only on $a_0,A_0,p,q$,
    \begin{equation}\label{eq:tail}
        c\frac{e^{-I_n}}{\sqrt n\lambda_n}
        \le
        \P(S_n<0)
        \le
        C\frac{e^{-I_n}}{\sqrt n\lambda_n};
    \end{equation}
    \item if $b_n=o(1/\lambda_n)$ and $b_n\to\infty$, then
    \begin{equation}\label{eq:small-window}
        \P(0\le S_n\le b_n)
        \le
        C\frac{b_n}{\sqrt n}e^{-I_n};
    \end{equation}
    \item if $x_n=o(\sqrt n)$ and $\lambda_nx_n\to\infty$, then
    \begin{equation}\label{eq:positive-window}
        \P(1\le S_n\le x_n)
        \ge
        c\frac{e^{-I_n}e^{\lambda_nx_n}}{\sqrt n\lambda_n}.
    \end{equation}
\end{enumerate}
The same conclusions hold after changing either binomial parameter count by $O(1)$.
\end{lemma}

\begin{proof}
Let $K_n(\theta)=(n-1)\log(1-p+pe^\theta)+(n+\Delta_n)\log(1-q+qe^{-\theta})$.  This formula also covers $p=1$, where the first term is simply $(n-1)\theta$.  Then $K_n'(0)=q\Delta'_n-p$ and $K_n''(0)=\tau^2n+O(\sqrt{n\log n})$, while $K_n'''(\theta)=O(n)$ uniformly for $|\theta|\le C\sqrt{\log n/n}$.  Let $\theta_n<0$ solve $K_n'(\theta_n)=0$ and set $\lambda_n=-\theta_n$.  Taylor expansion gives $\theta_n=-q\Delta'_n/(\tau^2 n)+O((\log n)/n)$ and $K_n(\theta_n)=-q^2(\Delta'_n)^2/(2\tau^2 n)+o(1)=-I_n+o(1)$.

Tilt the law by
\[
    \frac{d\P_{\theta_n}}{d\P}
    =
    \exp\{\theta_nS_n-K_n(\theta_n)\}.
\]
Under $\P_{\theta_n}$, the summands are independent bounded lattice variables, the total mean is $0$, the total variance is $\tau^2n(1+o(1))$, and the lattice span is $1$.  The standard local central limit theorem for triangular arrays of bounded lattice variables gives
\[
    \sup_k \P_{\theta_n}(S_n=k)\le \frac{C}{\sqrt n},
\]
and, uniformly for $k=o(\sqrt n)$,
\[
    \P_{\theta_n}(S_n=k)
    =
    \frac{1+o(1)}{\sqrt{2\pi\tau^2n}}.
\]
Changing measure back yields \eqref{eq:point}.

For $S_n<0$, the upper bound follows from $\sum_{j\ge1}e^{-\lambda_nj}\le C/\lambda_n$ and the tilted local upper bound.  The lower bound follows by summing over $1\le j\le\lfloor1/\lambda_n\rfloor$, where the local asymptotic is uniform.  This proves \eqref{eq:tail}.  If $0\le k\le b_n=o(1/\lambda_n)$, then $e^{\lambda_nk}=1+o(1)$, which gives \eqref{eq:small-window}.  For \eqref{eq:positive-window}, sum \eqref{eq:point} over $x_n-\lfloor1/\lambda_n\rfloor\le k\le x_n$.
The $O(1)$ stability follows because changing a binomial count by $O(1)$ changes $K_n(\theta_n)$ by $O(|\theta_n|)=o(1)$.
\end{proof}

As a consequence, we derive a concentration result on the number of vertices $v\in V^{(0)}_-$ with $S_v$ in a given window, through the second moment method.

\begin{lemma}\label{lem:counts}
Assume \eqref{eq:md-range}.  Let $J_n$ be either $(-\infty,-1]$ or $[1,x_n]$, where $x_n=o(\sqrt n)$ and $\lambda_nx_n\to\infty$.  Define $N(J_n)=\#\{v\in \Vm{0}:S_v\in J_n\}$ and $\pi_n=\P(S_v\in J_n)$.  If $n\pi_n\to\infty$, then $N(J_n)/(n\pi_n)\to1$ in probability.
\end{lemma}

\begin{proof}
For distinct $v,w\in \Vm{0}$, let $\xi$ be the indicator of the initial edge $\{v,w\}$.  Then $S_v=Z_v+\xi$ and $S_w=Z_w+\xi$, where $Z_v,Z_w,\xi$ are independent and $\xi\sim\operatorname{Bernoulli}(p)$.  Put $a_n=\P(Z_v\in J_n)$ and $b_n=\P(Z_v+1\in J_n)$.  Then $\pi_n=(1-p)a_n+pb_n$ and $\P(S_v\in J_n,S_w\in J_n)=(1-p)a_n^2+pb_n^2$.  Therefore the covariance is $p(1-p)(a_n-b_n)^2$.

If $J_n=(-\infty,-1]$, then $|a_n-b_n|\le \P(Z_v=-1)$.  By Lemma~\ref{lem:local}, this is $O(e^{-I_n}/\sqrt n)$, whereas $\pi_n\ge c e^{-I_n}/(\sqrt n\lambda_n)$.  Thus $|a_n-b_n|=o(\pi_n)$.  If $J_n=[1,x_n]$, the symmetric difference is contained in the two endpoints $Z_v=0$ and $Z_v=x_n$, so
\[
    |a_n-b_n|\le C\frac{e^{-I_n}e^{\lambda_nx_n}}{\sqrt n}.
\]
By \eqref{eq:positive-window}, $\pi_n\ge c e^{-I_n}e^{\lambda_nx_n}/(\sqrt n\lambda_n)$, so again $|a_n-b_n|=o(\pi_n)$.  Hence $\Var (N(J_n))\le n\pi_n+o(n^2\pi_n^2)$.
Since $n\pi_n\to\infty$, Chebyshev's inequality proves the claim.
\end{proof}

\section{Proofs of the main results}\label{sec:proofs}

\subsection{Sufficient conditions for consensus}

For the non-Markovian model, as mentioned previously, the condition $p>q$ ensures that there would be no change of opinion from $+$ to $-$ with high probability if we start with an advantage of the opinion $+$. To illustrate this fact, for $t\in\N$, we define $\sN_t:=\{V^{(t)}_-\cap V^{(t-1)}_+\neq\emptyset\}$, the event that  there exists an opinion change from $+$ to $-$ at time $t$. Let also $\sN:=\cup_{t\in\N}\sN_t$. 

\begin{proposition}
Consider the non-Markovian model on $ \SBM(n+\Delta,n,p,q) $ where $0<q<p\leq 1$ and $\Delta\in\N$. It holds that $\P(\sN)\to 0$.\label{prop:no+to-}
\end{proposition}

\begin{proof}
For a vertex $v$, let $d_+^{(t)}(v)$ and $d_-^{(t)}(v)$ be its
numbers of neighbors in $V_+^{(t)}$ and $V_-^{(t)}$, respectively, in
$G_t$. Put $A_t:=\cap_{s=1}^t\sN_s^c$, with $A_0=\Omega$, and let
$S_t:=V_-^{(t-1)}\cap V_+^{(t)}$.
Set
$
r_n:=\P\left(\Bin(n,q)>\Bin(n+\Delta-1,p)\right),
$
where the two binomial variables are independent. Thus
$
\P(\sN_1)\le (n+\Delta)r_n .
$
If $X\sim\Bin(n,q)$ and $Y\sim\Bin(n+\Delta-1,p)$, then
$$
\Exp[Y-X]=(n+\Delta-1)p-nq\ge c_0(n+\Delta)
$$
for some $c_0=c_0(p,q)>0$. Hence Hoeffding's inequality gives
\begin{align}
    r_n=\P(Y-X<0)\le \exp(-c(n+\Delta))\label{eq:rn bound}
\end{align}
for some $c=c(p,q)>0$.

Now fix $t\ge2$ and work on $A_{t-1}$. We claim that on $A_{t-1}$, $V^{(t)}_-\cap V^{(t-1)}_+\subseteq S_{t-1}$. Indeed, if
$v\in V_+^{(t-1)}\setminus S_{t-1}$, then $v$ was already positive at time
$t-2$. Since it stayed positive at time $t-1$,
$
d_+^{(t-2)}(v)\ge d_-^{(t-2)}(v).
$
On $A_{t-1}$, vertices only keep their opinions or move from $-$ to $+$.
Thus, from time $t-2$ to $t-1$, the positive degree of $v$ cannot
decrease and its negative degree cannot increase. Hence
$$
d_-^{(t-1)}(v)\le d_-^{(t-2)}(v)\le d_+^{(t-2)}(v)\le d_+^{(t-1)}(v),
$$
so $v$ cannot switch from $+$ to $-$ at time $t$, meaning that $v\not\in V^{(t)}_-$.

It remains to bound the newly switched vertices $S_{t-1}$. Conditional on the history
just before $G_{t-1}$ is formed, if $v\in S_{t-1}$, then all edges incident
to $v$ are resampled. On $A_{t-1}$,
$
|V_-^{(t-1)}|\le n\text{ and } |V_+^{(t-1)}|-1\ge n+\Delta-1.
$
Therefore, by stochastic domination,
$$
d_-^{(t-1)}(v)\le_{\rm st}\Bin(n,q)\qquad\text{and}\qquad
d_+^{(t-1)}(v)\ge_{\rm st}\Bin(n+\Delta-1,p).
$$ Hence
$
\P(v\in V_-^{(t)}\mid \mathcal F_{t-1})\le r_n .
$
Since $|S_{t-1}|\le n$ on $A_{t-1}$, the union bound gives
$
\P(\sN_t\cap A_{t-1})\le nr_n,~ t\ge2.
$

Finally, if no $+$ to $-$ change occurs in the first $n$ steps, then after
time $n$ either the dynamics has halted or $+$ has won. Therefore, by \eqref{eq:rn bound},
$$
\P(\sN)
\le \P(\sN_1)+\sum_{t=2}^n \P(\sN_t\cap A_{t-1})
\le (n+\Delta)r_n+n^2r_n
=o(1),
$$completing the proof.
\end{proof}


We first analyze some simple sufficient conditions for opinion $+$ to win in the next day.
\begin{proposition}\label{prop2lastdaywins}
 Let  $\delta>0$ and $t\in\N$. Recall that $\Delta'=(p-q)n/q-\Delta$. For the non-Markovian model with $\Delta'=o(n^{1/2+\delta})$, it holds that
$$\P(\sP_{t+1})\geq \P\left(|V^{(t-1)}_-\cap V^{(t)}_+|\geq \delta n^{1/2+\delta}\right)-o(1).$$
\end{proposition}

\begin{proof} For a vertex $v$, denote  by $v^{(t)}_+,v^{(t)}_-$ the number of neighbors of $v$ in $V^{(t)}_+,V^{(t)}_-$ respectively for $t\in\N$. First, we intersect with the event $\sN^c$ that there is no opinion change from $+$ to $-$ throughout the dynamics, thus removing a set of probability $o(1)$ by Proposition \ref{prop:no+to-}. Next, we intersect with the event 
$$\sE:=\left\{v^{(0)}_-\leq np+n^{(\delta+1)/2}\text{ and }v^{(0)}_+\geq (n+\Delta)q-n^{(\delta+1)/2}\text{ for any }v\in V^{(0)}_-\right\}.$$By \eqref{eq:Hoeffding} and the union bound, it holds that $\P(\sE^c)=o(1)$. Therefore, it suffices to prove
\begin{align}\P\left(\sP_{t+1}^c\cap\sE\cap(\bigcap_{s\leq t}\sN_s^c)\cap\{|V^{(t-1)}_-\cap V^{(t)}_+|\geq \delta n^{1/2+\delta}\}\right)= o(1).\label{eq:o(1)claim}\end{align}
Let $\F_t$ be the $\sigma$-algebra generated by $ \pth{G_s,\bfW_s}_{0\leq s\leq t-1} $, i.e., the first $t-1$ days of the dynamics. In the following, we condition on $\F_t$ and let $v\in V^{(t)}_-$. Denote by $\P_t$ the corresponding conditional probability. Note that $V^{(t)}_-$ is $\F_t$-measurable while $v^{(t)}_+$ is not. By definition, on the event $\sE\cap (\bigcap_{s\leq t}\sN_s^c)$ we have $$v^{(t-1)}_+\geq v^{(0)}_+\geq (n+\Delta)q-n^{(\delta+1)/2}\text{ and }v^{(t)}_-\leq v^{(0)}_-\leq np+n^{(\delta+1)/2}.$$
 On the other hand, on the event $\{|V^{(t-1)}_-\cap V^{(t)}_+|\geq \delta n^{1/2+\delta}\}\cap(\bigcap_{s\leq t}\sN_s^c)$, it holds that $$v^{(t)}_+\geq_{\rm st}v^{(t-1)}_++\Bin(\delta n^{1/2+\delta},q)\geq (n+\Delta)q-n^{(\delta+1)/2}+\Bin(\delta n^{1/2+\delta},q),$$where $\leq_{\rm st}$ means (first-order) stochastic dominance.
It follows that on the event $\sX:=\{|V^{(t-1)}_-\cap V^{(t)}_+|\geq \delta n^{1/2+\delta}\}\cap(\bigcap_{s\leq t}\sN_s^c)\cap \sE$, we have
\begin{align*}\P_t\left(\{v\not\in V_+^{(t+1)}\}\cap\sX\right)&=\P_t\left(\{v^{(t)}_+\leq v^{(t)}_-\}\cap \sX\right)\\
&\leq \P\left(\{(n+\Delta)q-n^{(\delta+1)/2}+\Bin(\delta n^{1/2+\delta},q)\leq np+n^{(\delta+1)/2}\}\cap \sX\right)\\
&\leq \exp\left(-2\delta q^2n^{(\delta+1)/2}\right),\end{align*}where we have used our assumption $\Delta'=o(n^{1/2+\delta})$ and \eqref{eq:Hoeffding} in the last inequality.
We finally conclude from a union bound that
$$\P_t\left(\sP_{t+1}^c\cap\sE\cap(\bigcap_{s\leq t}\sN_s^c)\cap\{|V^{(t-1)}_-\cap V^{(t)}_+|\geq \delta n^{1/2+\delta}\}\right)\leq n\exp\left(-2\delta q^2n^{(\delta+1)/2}\right).$$
Taking expectation yields the claim \eqref{eq:o(1)claim} and hence completes the proof.
\end{proof}

\begin{proposition}\label{prop1lastdaywins}Let $t\in\N_0$.
For the Markovian model,
$$\P(\sP_{t+1})\geq \P\left(|V^{(t)}_+|\geq \frac{p}{q}|V^{(t)}_-|+L\sqrt{|V^{(t)}_-|\log |V^{(t)}_-|}\right)-o(1).$$
\end{proposition}

\begin{proof} Suppose first that $p<1$. By the Markov property, we may assume $t=0$ and write $n=|V^{(t)}_-|$ and $$\Delta_n\geq \left(\frac{p-q}{q}\right)n+L_1\sqrt{n\log n}.$$
Using similar arguments as in Lemma \ref{lemmaBinomialtail2}, one can prove that by choosing $L_1$ above large enough,  for any $v\in V^{(0)}_-$,
$$\P\left(v\in V^{(1)}_-\right)\leq Ln^{-2}.$$
Using a union bound over the set $V^{(0)}_-$ finishes the proof.

It remains to treat the endpoint case $p=1$. In this case, every pair of vertices
with the same opinion is connected. Hence, for any $v\in V_-^{(t)}$, the number
of neighbors of $v$ in $V_-^{(t)}$ is deterministically $|V_-^{(t)}|-1$, while
the number of neighbors of $v$ in $V_+^{(t)}$ has distribution
$\Bin(|V_+^{(t)}|,q)$.

Assume that
\[
|V_+^{(t)}|\ge \frac{1}{q}|V_-^{(t)}|+L\sqrt{|V_-^{(t)}|\log |V_-^{(t)}|}.
\]
Then, for each $v\in V_-^{(t)}$,
\[
\mathbb E[d_+^{(t)}(v)]
=q|V_+^{(t)}|
\ge |V_-^{(t)}|+qL\sqrt{|V_-^{(t)}|\log |V_-^{(t)}|}.
\]
By Chernoff's inequality, choosing $L=L(q)$ sufficiently large gives
\[
\mathbb P\left(d_+^{(t)}(v)\le |V_-^{(t)}|-1\right)
\le |V_-^{(t)}|^{-2}.
\]
A union bound over all vertices in $V_-^{(t)}$ implies that, with probability
$1-o(1)$, every vertex in $V_-^{(t)}$ has a strict majority of $+$ neighbors and
therefore changes to opinion $+$ at time $t+1$. Since the set $V_+^{(t)}$ is nondecreasing, this completes the proof.
\end{proof}

Next, we show for the Markovian model that a sufficient lead of $\Omega(n)$ will guarantee a win a.a.s.

\begin{proposition}\label{prop:model2winseventually}
Assume $0<q<p<1$. Let $\delta>0$ be arbitrary, then for the Markovian model,  uniformly for $n,\Delta$ such that $\Delta>\delta n$, it holds 
$\P(\sP)=1-o(1).$
\end{proposition}

\begin{proof}
Recall that $\{|V^{(t)}_+|\}$ forms a Markovian random walk on $\{0,1,\dots,2n+\Delta\}$. We have shown in Proposition \ref{prop1lastdaywins} that once the random walk reaches $2n+\Delta-n/L$ then $\sP$ happens with high probability; thus it suffices if we show that starting from $n+\Delta$, the random walk is monotonically non-decreasing. This motivates the study of the following conditional probability. Writing $|V^{(0)}_-|=n$ and $|V^{(0)}_+|=n+\Delta$, we have
\begin{align*}p_{n,n+\Delta}:=\P\bigg(|V^{(0)}_+\cap V^{(1)}_-|\geq 1\big\vert |&V^{(0)}_+\cap V^{(1)}_-|+|V^{(0)}_-\cap V^{(1)}_+|\neq 0\bigg)\\ &=\frac{\P\left(|V^{(0)}_+\cap V^{(1)}_-|\geq 1\right)}{\P\left(|V^{(0)}_+\cap V^{(1)}_-|+|V^{(0)}_-\cap V^{(1)}_+|\neq 0\right)}\leq \frac{\P\left(|V^{(0)}_+\cap V^{(1)}_-|\geq 1\right)}{\P\left(|V^{(0)}_-\cap V^{(1)}_+|\geq 1\right)}.
\end{align*}
Using a union bound and similar arguments as in Lemma \ref{lemmaBinomialtail2}, we compute
\begin{align*}\P\left(|V^{(0)}_+\cap V^{(1)}_-|\geq 1\right)&\leq Ln\P(\Bin(n,q)>\Bin(n+\Delta-1,p))\\
&\leq Ln^2\max_{j\in\Z\cap[nq,(n+\Delta-1)p]}\P(\Bin(n,q)>j)~\P(\Bin(n+\Delta-1,p)<j)\\
&\leq Ln^2\exp(-C_1n)\end{align*}
and 
\begin{align*}
    \P\left(|V^{(0)}_-\cap V^{(1)}_+|\geq 1\right)&\geq \P(\Bin(n+\Delta,q)>\Bin(n-1,p))\\
    &\geq \max_{k\in\Z\cap[nq,(n+\Delta-1)p]}\P(\Bin(n+\Delta,q)>k)~\P(\Bin(n-1,p)<k)\\
    &\geq \frac{1}{L}\exp(-C_2n)
\end{align*}
where $C_1,C_2$ are constants that do not depend on $n$ and $C_2<C_1$ (here we use $\Delta>\delta n$). This shows $p_{n,n+\Delta}\leq L\exp(-n/L)$. Using a union bound shows that with probability $1-o(1)$, the random walk $\{|V^{(t)}_+|\}$ increases monotonically from $n+\Delta$ to $2n+\Delta-n/L$, completing the proof.
\end{proof}

\subsection{Proof of Theorem \ref{thmmodel1}}

We recall from \eqref{eq:prbpbr} that $p_{-+}=\prob{\Bin(n+\Delta,q)>\Bin(n-1,p) }$. It follows from Hoeffding's inequality that uniformly for $\Delta\in[0,n/L]$, it holds
\begin{align}
    p_{-+}\leq L\exp\left(-\frac{n}{L}\right).\label{eq:expsmall}
\end{align}
Next, we study the transition probabilities of the Markov chain $\{|V^{(t)}_+|\}_{t\in\N}$. Consider the Markovian model on $ \SBM(j,2n+\Delta-j,p,q) $ where $j\in[2n+\Delta]$. We denote by \begin{align}\begin{split}
    &p_r(j)=p_r(j;n,\Delta):=\P\left(|V^{(0)}_-\cap V^{(1)}_+|=1\right)\\ &\text{ and  }\ p_\ell(j)=p_\ell(j;n,\Delta):=\P\left(|V^{(0)}_+\cap V^{(1)}_-|=1\right),
\end{split}\label{eq:pj}\end{align}where $|V^{(0)}_+|=j$ and $|V^{(0)}_-|=2n+\Delta-j$. As a special case that corresponds to the first step of the Markovian model on $ \SBM(n+\Delta,n,p,q) $, we write
$$p_r=p_r(n+\Delta)=\P\left(|V^{(0)}_-\cap V^{(1)}_+|=1\right)\ \text{ and  }\ p_\ell=p_\ell(n+\Delta)=\P\left(|V^{(0)}_+\cap V^{(1)}_-|=1\right).$$

\begin{lemma}\label{lemma:probratio}
Assume $0<q<p<1$. Consider the Markovian model on $ \SBM(n+\Delta,n,p,q) $ where $n+\Delta=|V^{(0)}_+|> |V^{(0)}_-|=n$.
Then uniformly in  $n$ and $\Delta\leq n/L$, it holds that
\begin{align} \frac{p_r}{p_\ell} \geq 1+\frac{1}{L}, \label{eq:prpl}\end{align}
and
\begin{align}\frac{\P \pth{|V^{(0)}_-\cap V^{(1)}_+|\geq 2}+\P\pth{|V^{(0)}_+\cap V^{(1)}_-|\geq 2}}{p_r+p_\ell}\leq L\exp\left(-\frac{n}{L}\right).\label{eq:not>2}\end{align}
\end{lemma}

\begin{proof}
We label the vertices in $ \Vp{0} $ by $ i \in [n+\Delta] $ and vertices in $ \Vm{0} $ by $ j \in [n] $. Recall that the degree sequences of the graph $ \SBM(n+\Delta,n,p,q) $ are sampled from three independent random graphs $ \sfG(n+\Delta,p) $, $ \sfG(n,p) $ and $ \sfG(n+\Delta,n,q) $, and we denote by $ \bD_{+} $, $ \bD_{-} $, and $ (\bS,\bT) $ the degree sequences of these subgraphs, respectively.

Note that
\begin{multline*}
\prob{|\Vm{0}\cap\Vp{1}|=1} \leq \sum_{j=1}^n \prob{j \in \Vm{0}\cap\Vp{1}} = n \, \P \pth{\bT(1) > \bD_{-}(1)}\\ 
= n\, \prob{\Bin(n+\Delta,q) > \Bin(n-1,p)} = n p_{-+},
\end{multline*}
where $ p_{-+} $ is defined in \eqref{eq:prbpbr}.
On the other hand, by the union bound,
\begin{align*}
\prob{|\Vm{0}\cap\Vp{1}|=1} &= \sum_{j=1}^n \prob{ j \in \Vm{0}\cap\Vp{1},\ \mbox{and}\ k \notin \Vm{0}\cap\Vp{1}  \text{ for all } k \neq j}\\
&\geq \sum_{j=1}^n \Bigg( \prob{j \in \Vm{0}\cap\Vp{1}} - \sum_{\substack{1\leq k\leq n\\ k \neq j}} \prob{j,k \in \Vm{0}\cap\Vp{1}} \Bigg)\\
&= \sum_{j=1}^n \prob{ j \in \Vm{0}\cap\Vp{1}} - \sum_{\substack{1\leq j,k\leq n\\ k\neq j}} \prob{j,k \in \Vm{0}\cap\Vp{1}}.
\end{align*}
To estimate the second term, we observe that for $j\neq k$,
$$ \prob{j,k \in \Vm{0}\cap\Vp{1}} \leq \tprob{j,k \in \Vm{0}\cap\Vp{1}} $$
where $ \widetilde{\P} $ is the probability measure of the graph $ G \sim \SBM(n+\Delta,n,p,q) $ conditioned on $ (j,k) \notin E $. Let $ \widetilde{\bD}_+ $, $ \widetilde{\bD}_- $, and $ (\widetilde{\bS},\widetilde{\bT}) $ denote the degree sequences with respect to $ \widetilde{\P} $. Then,
\begin{align*}
\widetilde{\P} \pth{j,k \in \Vm{0}\cap\Vp{1}} &\leq \widetilde{\P} \pth{ \widetilde{\bT}(j) > \widetilde{\bD}_{-}(j) }~ \widetilde{\P} \pth{ \widetilde{\bT}(k) > \widetilde{\bD}_{-}(k) }\\
&\leq \pth{\prob{\Bin(n+\Delta,q) > \Bin(n-2,p)}}^2\\
&\leq L \pth{\prob{\Bin(n+\Delta,q) > \Bin(n-1,p)}}^2\\
&\leq L p_{-+}^2.
\end{align*}
This implies that
\begin{equation*}
\prob{|\Vm{0}\cap\Vp{1}|=1} \geq n p_{-+} - L n^2 p_{-+}^2 = (1-o(1)) n p_{-+},
\end{equation*}
where the last step follows from \eqref{eq:expsmall}.
Therefore, we conclude that
$$ p_r = \prob{|\Vm{0}\cap\Vp{1}|=1} = (1-o(1)) n p_{-+}, $$
and a similar argument yields
$$ p_\ell = (1-o(1)) (n+\Delta) p_{+-}. $$
It follows that
$$ \frac{p_r}{p_\ell} \geq (1-o(1)) \frac{n p_{-+}}{(n+\Delta) p_{+-}} \geq (1-o(1)) \pth{1-\frac{1}{L_2+1}} \frac{p_{-+}}{p_{+-}} $$
where we used $\Delta\leq n/L_2$ with $L_2$ a large constant to be determined.

Consider the independent random variables $ T \sim \Bin(n+\Delta,q) $, $ S \sim \Bin(n,q) $, $ D_- \sim \Bin(n-1,p) $ and $ D_+ \sim \Bin(n+\Delta-1,p) $. By choosing $L_2$ above large enough, in order to prove \eqref{eq:prpl}, it suffices to show
\begin{equation}\label{eq:Ratio_p_rp_l}
 \frac{\prob{T>D_-}}{\prob{S>D_+}} \geq 1+\frac{1}{L}.
\end{equation}

Since $\Bin(n-1,p)\leq_{\rm st} \Bin(n+\Delta-1,p)$, 
we may replace $D_-$ in the numerator by $D_+$. Similarly, using stochastic dominance again, we may replace $ T $ with $ T' \sim \Bin(n+1,q) $. By a suitable coupling, we may assume
$T'=S+\epsilon$  where $\epsilon\sim\Bin(1,q)$ is independent from $S$, and $(S,T')$ and $D_+$ are independent. It follows that
\begin{equation}\label{eq:RatioReduction}
\frac{\prob{T>D_-}}{\prob{S>D_+}}-1 \geq \frac{\prob{T>D_+}}{\prob{S>D_+}}-1 \geq \frac{\prob{T' > D_+}}{\prob{S>D_+}}-1 = \frac{\prob{S=D_+,\epsilon=1}}{\prob{S>D_+}} = \frac{q \, \prob{S=D_+}}{\prob{S>D_+}}.
\end{equation}
To estimate the right-hand side of \eqref{eq:RatioReduction}, note that
\begin{align*}
\frac{\prob{S=D_+}}{\prob{S=D_{+}+1}} = \frac{ \sum_{x=0}^{n+\Delta-1} \prob{D_+=x} \prob{S=x} }{ \sum_{x=0}^{n+\Delta-1} \prob{D_+=x} \prob{S=x+1} }.
\end{align*}
Using log-concavity of the functions $x\mapsto \prob{D_+=x} \prob{S=x} $ and $x\mapsto \prob{D_+=x} \prob{S=x+1} $ and that $\Delta\leq n/L_2$, it is not hard to show that there exist $ q_1,q_2 $ depending only on $p,q$ satisfying $ nq < nq_1 < nq_2 < (n+\Delta)p $ so that
$$ \frac{ \sum_{x=0}^{n+\Delta-1} \prob{D_+=x} \prob{S=x} }{ \sum_{x=0}^{n+\Delta-1} \prob{D_+=x} \prob{S=x+1} } = (1+o(1)) \frac{ \sum_{x=nq_1}^{nq_2} \prob{D_+=x} \prob{S=x} }{ \sum_{x=nq_1}^{nq_2} \prob{D_+=x} \prob{S=x+1} }, $$
and for some $L_3>0$,
$$ \frac{\prob{S=x}}{\prob{S=x+1}} = (1+o(1)) \frac{(1-q)x}{q(n-x)} \geq 1+\frac{1}{L_3},\ \ \mbox{for}\ x \in [nq_1,nq_2]\cap\Z.$$This implies that
\begin{equation}\label{eq:S_Decay}
\frac{\prob{S=D_+}}{\prob{S=D_{+}+1}} \geq 1+\frac{1}{L_3}.
\end{equation}
Recall that the probability mass function of a binomial distribution is log-concave, and that a convolution of log-concave functions is log-concave. This implies the probability mass function of $S-D_+$ is log-concave. Combined with \eqref{eq:S_Decay}, we obtain
\begin{align*} \prob{S>D_+} \leq \sum_{x=1}^{n+\Delta-1} \prob{S-D_+=x} &\leq \sum_{x=1}^{n+\Delta-1} \pth{\frac{L_3}{L_3+1}}^{x-1} \prob{S=D_++1} \\
&\leq (L_3+1) \prob{S=D_++1}. \end{align*}
Consequently, we have
$$ \frac{\prob{S=D_+}}{\prob{S>D_+}} \geq \frac{\prob{S=D_+}}{(L_3+1) \prob{S=D_++1}} \geq \frac{1}{L}. $$
Plugging back into \eqref{eq:RatioReduction}  yields the desired \eqref{eq:Ratio_p_rp_l}.

To show \eqref{eq:not>2}, by the union bound, for $i\neq j$,
$$ \prob{|\Vm{0}\cap\Vp{1}| \geq 2} \leq n^2 \, \prob{i,j \in \Vm{0}\cap\Vp{1}}, $$
and recall that for $i\neq j$,
$$ \prob{i,j \in \Vm{0}\cap\Vp{1}} \leq L p_{-+}^2. $$
Therefore, it holds that
$$ \frac{\prob{|\Vm{0}\cap\Vp{1}| \geq 2}}{p_r} \leq \frac{L n^2 p_{-+}^2}{(1-o(1))n p_{-+}} \leq L \exp \pth{-\frac{n}{L}}, $$
where the last step follows from \eqref{eq:expsmall}.
Similarly, we also have
$$ \frac{\prob{|\Vp{0}\cap\Vm{1}| \geq 2}}{p_\ell} \leq L \exp \pth{-\frac{n}{L}}. $$
Combining these together yields the desired result.
 \end{proof}

\begin{proof}[Proof of Theorem \ref{thmmodel1}] 
It is clear that if $p=1$, $|V_+^{(t)}|$ is nondecreasing and $\P(\sP)=1$, so in the rest of this proof we assume $p<1$.

Part (\ref{1i}). The proof gives the asserted bound for all $n\ge n_0(p,q)$, uniformly in $\Delta\ge 1$.
For the finitely many values $n<n_0(p,q)$, the same lower bound follows after increasing $L$,
since for each fixed $n$, the probability is strictly larger than $1/2$ for every
$\Delta\ge 1$ and is increasing in $\Delta$ by standard coupling arguments.

 We fix $p,q\in [0,1]$ with $0<q<p\leq 1$ and $\Delta=1$, where the general case $\Delta\geq 1$ can be established using exactly the same proof with the same constant $L$. In this case, $\{|V^{(t)}_+|\}$ is a Markovian random walk on $\{0,1,\dots,2n+1\}$ whose transition probabilities are symmetric along $n+1/2$.

Recall \eqref{eq:pj}. By Lemma \ref{lemma:probratio}, there exists $L_2$ such that the probabilities $p_r(j)$ and $p_\ell(j)$ at point $j\in[n+1,n+n/L_2]\cap\Z$ satisfy $p_r(j)/p_\ell(j)\geq 1+1/L$ uniformly. It follows from Proposition \ref{prop:model2winseventually} that given the random walk $\{|V^{(t)}_+|\}$ reaches $n+n/L_2$, $\sP$ happens with probability $1-o(1)$. Thus by Markov property and symmetry of the random walk, it suffices to prove that $\{|V^{(t)}_+|\}$ reaches $n+n/L_2$ before $n$ with probability at least $1/L$. Define $\tau$ as the hitting time of the random walk with boundaries $n+n/L_2$ and $n$. Note that for $t<\tau$ on the event $A_t:=\{|V^{(t)}_-\cap V^{(t+1)}_+|\leq 1\text{ and }|V^{(t)}_+\cap V^{(t+1)}_-|\leq 1\}$, 
$$\frac{\P\left(|V^{(t+1)}_+|-|V^{(t)}_+|=1\right)}{\P\left(|V^{(t+1)}_+|-|V^{(t)}_+|=-1\right)}= \frac{\P\left(|V^{(t)}_-\cap V^{(t+1)}_+|=1\text{ and }|V^{(t)}_+\cap V^{(t+1)}_-|=0\right)}{\P\left(|V^{(t)}_-\cap V^{(t+1)}_+|=0\text{ and }|V^{(t)}_+\cap V^{(t+1)}_-|=1\right)}\geq \frac{p_r(|V^{(t)}_+|)}{p_\ell(|V^{(t)}_+|)}\geq 1+\frac{1}{L}.$$

We say that the random walk performs a non-trivial move at time $t$ if $|V^{(t+1)}_+|\neq |V^{(t)}_+|$. Let $\sigma_0=0$ and, recursively, let $\sigma_k$ be the time of the $k$-th non-trivial move. Consider the embedded chain $Y_k:=|V^{(\sigma_k)}_+|$, and let $\tau_*$ be the hitting time of the embedded chain with boundaries $n+n/L_2$ and $n$. Define the event
$$A=\bigcap_{\left\{k< \tau_*\right\}} A_{\sigma_k}.$$
This means that before the embedded chain is stopped, every non-trivial move has step length one. Taking intersection with the event $A$, it follows from a Gambler's Ruin argument (e.g., \cite[Theorem 1]{chong2000ruin}) applied to the embedded chain that $\P(Y_{\tau_*}=n+n/L_2)\geq 1/L$. It then suffices to prove that $\P(A)=1-o(1)$. By Lemma \ref{lemma:probratio} and a union bound over embedded moves, with probability $1-o(1)$, each step length of the first $n^3$ non-trivial moves of $\{|V^{(t)}_+|\}$ is one. Intersecting on this event, the probability that $\tau_*<n^3$, hence $A$ occurs, is $1-o(1)$ due to standard hitting-time estimates for the embedded nearest-neighbor walk (e.g., \cite{feller2008introduction}). This completes the proof of (\ref{1i}).

Part (\ref{1ii}). Assume $\Delta_n\to\infty$ and pick another sequence  $\kappa(n)\to\infty$ such that $\kappa(n)=o(\Delta_n)$. Let $L_2$ be given as in Lemma \ref{lemma:probratio}. By Proposition \ref{prop:model2winseventually}, it suffices to show that the Markovian random walk $\{|V^{(t)}_+|\}$ reaches $n+\Delta+n/L_2$ before $n+\Delta-\kappa(n)$ with probability tending to one. Using similar arguments as in the proof of (\ref{1i}), we may intersect with the event that when the random walk moves, it is either the case  $|V^{(t)}_-\cap V^{(t+1)}_+|=1\text{ and }|V^{(t)}_+\cap V^{(t+1)}_-|=0$, or the case $|V^{(t)}_-\cap V^{(t+1)}_+|=0\text{ and }|V^{(t)}_+\cap V^{(t+1)}_-|=1$. Since 
$$\frac{\P\left(|V^{(t)}_-\cap V^{(t+1)}_+|=1\text{ and }|V^{(t)}_+\cap V^{(t+1)}_-|=0\right)}{\P\left(|V^{(t)}_-\cap V^{(t+1)}_+|=0\text{ and }|V^{(t)}_+\cap V^{(t+1)}_-|=1\right)}\geq 1+\frac{1}{L}$$ uniformly given $n+\Delta-\kappa(n)\leq |V^{(t)}_+|\leq n+\Delta+n/L_2$, a standard Gambler's Ruin estimate (e.g., \cite[Theorem 1]{chong2000ruin}) completes the proof.
\end{proof}

\subsection{Proof of Theorem \ref{thmmodel2}}

For $v\in \Vm{0}$, define $S_v:=d_-^{(0)}(v)-d_+^{(0)}(v)$.  For a fixed $v$, $S_v\stackrel d=Y_n-X_n$, where $Y_n\sim\Bin(n-1,p)$ and $X_n\sim\Bin(n+\Delta_n,q)$ are independent.  Let $A_t:=\Vm{0}\cap \Vp{t}$ be the initially negative vertices that have become plus by day $t$, and let $\sN:=\{\exists t\ge1:\ \Vp{t-1}\cap \Vm{t}\ne\varnothing\}$. Assume that $\Delta_n\in\N$. 
By Proposition \ref{prop:no+to-},
\begin{equation}\label{eq:no-back}
    \P(\sN)\to0.
\end{equation}
We will implicitly intersect with the event $\sN^c$ throughout the proofs below.

\begin{lemma}\label{lem:reservoir}
Assume that $\Delta_n\in\N$. For $a=a(n)>0$, set $R_a:=\{v\in \Vm{0}:S_v<2a\}$.  If $\sN^c$ occurs and there exists a sequence $a=a(n)=o(n)$ such that $|R_a|<a$, then the dynamics halts without reaching consensus.
\end{lemma}

\begin{proof}
Assume $\sN^c$ and fix a sequence $a=a(n)=o(n)$ such that $|R_a|<a$.
On $\sN^c$, $A_t$ is increasing.  If $v\in \Vm{0}\setminus A_t$, then relative to day $0$ only vertices in $A_t$ have changed opinion.  Each changed vertex can decrease the margin $d_-(v)-d_+(v)$ by at most $2$: a present negative-neighbor edge can become a present positive-neighbor edge.  Thus the current margin of $v$ at day $t$ is at least $S_v-2|A_t|$.  If $v$ flips from minus to plus at the next update, its current margin is strictly negative, so $S_v<2|A_t|$.

Every vertex in $A_1$ has $S_v<0$ and hence lies in $R_a$.  Inductively, while $|A_t|<a$, every newly flipped vertex lies in $R_a$.  If $|A_t|$ first reached $a$, the first reaching set would be a subset of $R_a$, contradicting $|R_a|<a$.  Hence $|A_t|<a$ for all $t$.

Since $a=o(n)$, plus cannot absorb all initially negative vertices.  On $\sN^c$, minus cannot win because all initially positive vertices remain plus.  Since $A_t$ is an increasing sequence of subsets of the finite set $R_a$, it stabilizes.  Once it stabilizes, no opinions change and no edges are resampled, so the process halts.
\end{proof}


\begin{proof}[Proof of Theorem \ref{thmmodel2}]

\sloppy Part (\ref{2i}). Assume \eqref{eq:i}, equivalently $$\Delta'_n\ge H\sqrt{n\log n}+\varepsilon\frac{\sqrt n\log\log n}{\sqrt{\log n}}.$$ 
If $\Delta'_n>2H\sqrt{n\log n}$, let $a_n=\sqrt n/\log n$.  Then $\E[S_v]\ge2\tau\sqrt{n\log n}(1+o(1))$.  Since $\Delta_n\le (p-q) n/q$, we have  $n+\Delta_n\le pn/q$, and therefore $\Var(S_v)\le \tau^2n+O(1)$.  Bernstein's inequality gives $\P(S_v<2a_n)\le n^{-4/3}$
for all sufficiently large $n$, and hence the expected number of vertices with $S_v<2a_n$ is $o(a_n)$.  Lemma~\ref{lem:reservoir} and \eqref{eq:no-back} imply halting.

It remains to consider $$H\sqrt{n\log n}+\varepsilon\frac{\sqrt n\log\log n}{\sqrt{\log n}}\le\Delta'_n\le2H\sqrt{n\log n}.$$  Choose $0<\rho<\varepsilon/H$. Recall \eqref{eq:defs}. Since $C'H^2=1/2$, expanding the square gives $I_n\ge \frac12\log n+\rho\log\log n$ for all sufficiently large $n$.  Define $a_n=\sqrt n/(\sqrt{\log n}(\log n)^{\rho/2})$ and $R_n=\{v\in \Vm{0}:S_v<2a_n\}$.
By Lemma~\ref{lem:local},
\[
    \P(S_v<0)\le Cn^{-1/2}(\log n)^{-\rho-1/2}
\]
and
\[
    \P(0\le S_v<2a_n)
    \le C\frac{a_n}{\sqrt n}n^{-1/2}(\log n)^{-\rho}.
\]
Therefore,
\[
    \frac{\E[|R_n|]}{a_n}
    \le
    C(\log n)^{-\rho/2}+C(\log n)^{-\rho}\to0.
\]
Markov's inequality gives $|R_n|<a_n$ a.a.s.  Since $a_n=o(n)$, Lemma~\ref{lem:reservoir} and \eqref{eq:no-back} prove halting a.a.s.

Part (\ref{2ii}).  This follows directly from Proposition \ref{prop1lastdaywins}, since the first step is the same for both Markovian and non-Markovian models. 

Part (\ref{2iii}).  We may assume by (\ref{2ii}) that $$\Delta_n\leq \left(\frac{p-q}{q}\right)n+L\sqrt{n\log n}.$$ By Proposition \ref{prop2lastdaywins}, it suffices to prove that for some $\delta>0$, 
\begin{align}
    \P(|V^{(0)}_-\cap V^{(1)}_+|\leq \delta n^{(\delta+1)/2})=o(1).\label{eq:tp}
\end{align}

Suppose first that $p<1$. By Lemma \ref{lem:True_to_Conditioned2}, it suffices to show for some $\delta>0$ that for all $r_2\in 
   [0,1]\cap[p-L(\log n)/n,p+L(\log n)/n]$, $$\P_{\substack{\calB_{r_2}^n, \calB_{q}^{n+\Delta,n} }}\left(|V^{(0)}_-\cap V^{(1)}_+|\leq \delta n^{(\delta+1)/2}\right)=o(1).$$ 
   Recall our assumption that for some $\delta_0>0$,
   $$\Delta'=\frac{(p-q)n}{q}-\Delta\leq \sqrt{\frac{(1/2-\delta_0)n\log n}{C'}}.$$If $\Delta'<0$, the probability of a first-day flip from $-$ to $+$ is bounded below by the same probability with $\Delta'=0$, so the following estimate applies after replacing $\Delta'$ by $0$. If $\Delta'\in\N$, Lemma \ref{lemmaBinomialtail2} applies directly. In either case,
   $$\P(\Bin(n-1,r_2)<\Bin(n+\Delta,q))\geq \frac{1}{L} n^{(\delta_0-1)/2}.$$
   Thus by picking $\delta$ small enough, we have
   \begin{align*}\P_{\substack{\calB_{r_2}^n, \calB_{q}^{n+\Delta,n} }}\Bigg(|V^{(0)}_-\cap V^{(1)}_+|&\leq \delta n^{(\delta+1)/2}\Bigg)\\
       &=\P\left(\Bin\left(n,\P(\Bin(n-1,r_2)<\Bin(n+\Delta,q))\right)\leq \delta n^{(\delta+1)/2}\right)=o(1).
   \end{align*}
  This concludes the proof.

When $p=1$, a vertex in $V^{(0)}_-$ flips on the first day whenever $\Bin(n+\Delta,q)>n-1$. If $\Delta'<0$, where now $\Delta'=((1-q)/q)n-\Delta$, this probability is bounded below by the corresponding probability at $\Delta'=0$. Otherwise, Lemma \ref{lemmaBinomialtail1} implies that for some $\delta_2>0$,
$$\P(\Bin(n+\Delta,q)>n-1)\geq n^{-1/2+\delta_2}$$
for all sufficiently large $n$. Hence $|V^{(0)}_-\cap V^{(1)}_+|$ stochastically dominates a binomial random variable with mean at least $n^{1/2+\delta_2}$, and \eqref{eq:tp} follows from a Chernoff bound by choosing $\delta>0$ sufficiently small.

Part (\ref{2iv}). Assume \eqref{eq:v}, or equivalently $$\Delta'_n\le H\sqrt{n\log n}-\varepsilon\frac{\sqrt n\log\log n}{\sqrt{\log n}}.$$  If $\Delta'_n\le (H/2)\sqrt{n\log n}$, part (\ref{2iii}) gives $\sP_2$ a.a.s.  Hence, assume $$\frac{H}{2}\sqrt{n\log n}\le\Delta'_n\le H\sqrt{n\log n}-\varepsilon\frac{\sqrt n\log\log n}{\sqrt{\log n}}.$$  Choose $0<\rho<\varepsilon/H$.  Then $I_n\le\frac12\log n-\rho\log\log n$ for all large $n$.

Recall that $A_1=\Vm{0}\cap\Vp{1}=\{v\in \Vm{0}:S_v<0\}$.
By Lemmas~\ref{lem:local} and \ref{lem:counts},
\[
    |A_1|\ge m_n:=\frac{\sqrt n}{\sqrt{\log n}}(\log n)^\beta
\]
a.a.s. for any fixed $0<\beta<\min\{\rho,1/2\}$.  Note that $m_n\gg\log n$.

Define $B_n=\{v\in \Vm{0}:1\le S_v\le qm_n/4\}$.
For $x_n=qm_n/4$, we have $x_n=o(\sqrt n)$ and $\lambda_nx_n\asymp(\log n)^\beta\to\infty$.  Lemmas~\ref{lem:local} and \ref{lem:counts} imply that, for every fixed $M>0$, $|B_n|\ge M\sqrt{n\log n}$ a.a.s., because
\[
    \E[|B_n|]
    \ge
    c\sqrt n(\log n)^{\rho-1/2}\exp(c(\log n)^\beta).
\]

Condition on the initial graph and on $A_1$.  For $v\in B_n$, let $\xi_v$ be the number of fresh resampled edges from $A_1$ to $v$ after the day-one opinion update.  On $|A_1|\ge m_n$, $\xi_v\geq_{\rm st}\Bin(m_n,q).$ 
Chernoff's inequality and $m_n\gg\log n$ imply, by a union bound over all $v$, that $\xi_v\ge qm_n/2$ for every $v\in B_n$ a.a.s.  The current margin of $v$ after day one is at most $S_v-\xi_v$, because the old edges from $A_1$ can only further reduce the margin.  Thus every $v\in B_n$ has negative margin and flips on day two.  Therefore, for every fixed $M>0$,
\[
    |C_2|:=|\Vm{1}\cap \Vp{2}|
    \ge M\sqrt{n\log n}
\]
a.a.s.

Finally, choose $K=K(p,q)$ such that $\max_{v\in \Vm{0}}S_v\le K\sqrt{n\log n}$ a.a.s.  This follows from Bernstein's inequality and a union bound, since in the present range $\E[S_v]\le\tau\sqrt{n\log n}+o(\sqrt{n\log n})$ and $\Var(S_v)=O(n)$.  Choose $M$ so large that $qM/2>K$.  Conditional on the history up to day two, each remaining negative vertex receives independent fresh cross-opinion edges from $C_2$.  If $\eta_v$ is the number of such edges, then $\eta_v\geq_{\rm st}\Bin(|C_2|,q),$ 
and Chernoff's inequality plus a union bound gives
\[
    \eta_v\ge\frac q2|C_2|>K\sqrt{n\log n}\ge S_v
\]
for every remaining negative vertex.  Therefore, every remaining negative vertex flips on day three.  Intersecting with $\sN^c$, no plus vertex flips to minus, so plus wins by day three.

Part (\ref{2v}) is a direct consequence of Proposition \ref{prop:no+to-}. 
\end{proof}

\section{Numerical Experiments}\label{sec:Num}

In this section, we present some numerical experiments to validate our main theorems and illustrate the sharp phase transition in the non-Markovian model. Throughout this section, we run simulations of the dynamics and check the frequencies of various outcomes. Recall that $ n $ represents the number of vertices with opinion $ - $ in the initialization, and $ \Delta $ is the initial bias.


\subsection{Markovian model}

For the Markovian model, we first focus on the power-of-one phenomenon. In Table \ref{tab:delta1}, we run the Markovian model with $\Delta=1, p=0.5, q=0.3$ and let the initial number of vertices with opinion $ - $ increase. 



\begin{table}[!htb]
    \centering{\scalebox{0.9}{
    \begin{tabular}{|c|cccccccc|}
     \hline
     \textbf{Number of initial ``$-$''} & 50 &100 & 125 &150 & 175 & 200 & 225 & 250\\ 
     \hline
      \textbf{Winner} & ``$+$'' & ``$+$''& ``$+$''& ``$+$''& ``$+$''& ``$+$''& ``$+$''& ``$+$''\\ \hline
      \textbf{Average last day} & 7.68 & 26.90 & 62.21 & 146.36 & 380.47 & 1025.91 & 2757.31 & 6152.87 \\ \hline 
      \textbf{Frequency} & 717 & 677 & 679 & 679 & 687 & 625 & 656 & 633\\ 
      \hline 
    \end{tabular}}}
    \caption{Simulation of the Markovian model when $\Delta=1, p=0.5, q=0.3$. \label{tab:delta1}
    }
    \end{table}


\noindent As shown in Table \ref{tab:delta1}, the frequency of the opinion + winning is greater than half of all instances uniformly. This agrees with Part (\ref{1i}) in Theorem \ref{thmmodel1}, that is, a single initial bias already leads to a non-trivial advantage for winning in the end.

Next, recall that Part (\ref{1ii}) of Theorem \ref{thmmodel1} states that any initial bias of $ \omega(1) $ will guarantee an asymptotically almost sure win. To verify this, we simulate the dynamics with $ p=0.5,q=0.3 $ and choose two initial biases with different growth rates $ \Delta(n)=\lceil\log n\rceil $ and $ \Delta(n)=\lceil n/10 \rceil $. The results are listed below.

    \begin{table}[!htb]
    \centering{\scalebox{0.9}{
    \begin{tabular}{|c|cccccccc|}
     \hline
     \textbf{Number of initial ``$-$''} & 50 &100 & 125 &150 & 175 & 200 & 225 & 250\\ 
     \hline
     \textbf{$\bm{\Delta=\lceil\log n\rceil}$} & 4 & 5 & 5 & 6 & 6 & 6 & 6 & 6\\ 
     \hline
      \textbf{Winner} & ``$+$'' & ``$+$''& ``$+$''& ``$+$''& ``$+$''& ``$+$''& ``$+$''& ``$+$''\\ \hline
      \textbf{Average last day} & 5.79 & 16.17 & 35.30 & 70.38 & 183.80 & 445.94 & 1288.29 & 3463.56
      \\ \hline 
      \textbf{Frequency} & 982 & 989 & 988 & 998 & 994 & 999 & 991 & 990\\ 
      \hline 
    \end{tabular}}}
    \caption{Simulation of the Markovian model when $\Delta=\lceil \log n \rceil, p=0.5, q=0.3$. \label{tab:delta_ln}
    }
    \end{table}
    \begin{table}[!htb]
    \centering{\scalebox{0.9}{
    \begin{tabular}{|c|cccccccc|}
     \hline
     \textbf{Number of initial ``$-$''} & 50 &100 & 125 &150 & 175 & 200 & 225 & 250\\ 
     \hline
     {$\bm{\Delta=\lceil n/10 \rceil}$} & 5 & 10 & 13 & 15 & 18 & 20 & 23 & 25\\ 
     \hline
      \textbf{Winner} & ``$+$'' & ``$+$''& ``$+$''& ``$+$''& ``$+$''& ``$+$''& ``$+$''& ``$+$''\\ \hline
      \textbf{Average last day} & 3.73 & 9.75 & 14.46 & 24.58 & 41.39 & 78.69 & 144.20 & 288.41
      \\ \hline 
      \textbf{Frequency} & 1000& 1000& 1000& 1000& 1000& 1000& 1000& 1000\\ 
      \hline 
    \end{tabular}}}
    \caption{Simulation of the Markovian model when $\Delta=\lceil n/10 \rceil, p=0.5, q=0.3$. \label{tab:delta_c}
    }
    \end{table}
    
\noindent In Table \ref{tab:delta_ln}, the frequency of opinion $ + $ winning is greater than 98\% of all outcomes. Similarly, in Table \ref{tab:delta_c}, all simulations end up with opinion $ + $ winning. These confirm the theoretical result in Part (\ref{1ii}) of Theorem \ref{thmmodel1}.

Moreover, from the above results, we see that though a greater $ \Delta $ reduces the averaged time needed for consensus, the consensus time still exhibits an exponential growth in $ n $. This confirms our statement in Remark \ref{rmk:time}.

\subsection{Non-Markovian model}

We next run larger-size simulations for the non-Markovian model. For each listed
value of $n$, we set
$$ \Delta(n)= \left\lceil \pth{\frac{p-q}{q}}n - L \sqrt{n \log n} \right\rceil, $$
with varying constants $ L>0 $. As in Theorem \ref{thmmodel2}, we focus on the outcomes near
the critical value
$$ L^*=H(p,q)=\frac{\sqrt{p(2-p-q)}}{q}. $$
In our simulations, the last day of the dynamics is averaged over all instances
where consensus is reached. The frequency is given in the format $a/b$, meaning
that opinion $+$ wins $a$ times and the dynamics halts $b$ times.

\begin{table}[h]
\centering{\scalebox{0.74}{
\begin{tabular}{|cccccc||cccccc|}
\hline
\multicolumn{6}{|c||}{\textbf{Parameter block 1}} & \multicolumn{6}{c|}{\textbf{Parameter block 2}}\\ \hline
$\bm{n}$ & $\bm{L}$ & $\bm{\Delta}$ & \textbf{Winner} & \textbf{Last day} & \textbf{Frequency} &
$\bm{n}$ & $\bm{L}$ & $\bm{\Delta}$ & \textbf{Winner} & \textbf{Last day} & \textbf{Frequency}\\ \hline
1000 & 2.5 & 2126 & ``$+$''/Halt & 4.58 & 997/3 & 10000 & 2.5 & 22575 & ``$+$''/Halt & 4.02 & 1000/0\\ \hline
1000 & 2.7 & 2109 & ``$+$''/Halt & 6.31 & 846/154 & 10000 & 2.7 & 22514 & ``$+$''/Halt & 5.81 & 998/2\\ \hline
1000 & 2.788 & 2102 & ``$+$''/Halt & 7.23 & 640/360 & 10000 & 2.788 & 22488 & ``$+$''/Halt & 7.80 & 927/73\\ \hline
1000 & 2.8 & 2101 & ``$+$''/Halt & 7.41 & 596/404 & 10000 & 2.8 & 22484 & ``$+$''/Halt & 8.00 & 863/137\\ \hline
1000 & 3.0 & 2084 & ``$+$''/Halt & 9.36 & 69/931 & 10000 & 3.0 & 22423 & ``$+$''/Halt & 13.88 & 17/983\\ \hline
2000 & 2.5 & 4359 & ``$+$''/Halt & 4.39 & 1000/0 & 20000 & 2.5 & 45555 & ``$+$''/Halt & 4.00 & 1000/0\\ \hline
2000 & 2.7 & 4334 & ``$+$''/Halt & 6.27 & 920/80 & 20000 & 2.7 & 45466 & ``$+$''/Halt & 5.58 & 1000/0\\ \hline
2000 & 2.788 & 4323 & ``$+$''/Halt & 7.54 & 719/281 & 20000 & 2.788 & 45426 & ``$+$''/Halt & 7.64 & 974/26\\ \hline
2000 & 2.8 & 4322 & ``$+$''/Halt & 7.70 & 707/293 & 20000 & 2.8 & 45421 & ``$+$''/Halt & 8.04 & 949/51\\ \hline
2000 & 3.0 & 4297 & ``$+$''/Halt & 10.88 & 69/931 & 20000 & 3.0 & 45332 & ``$+$''/Halt & 15.50 & 2/998\\ \hline
5000 & 2.5 & 11151 & ``$+$''/Halt & 4.12 & 1000/0 & 50000 & 2.5 & 114828 & ``$+$''/Halt & 4.00 & 1000/0\\ \hline
5000 & 2.7 & 11110 & ``$+$''/Halt & 6.05 & 982/18 & 50000 & 2.7 & 114681 & ``$+$''/Halt & 5.33 & 1000/0\\ \hline
5000 & 2.788 & 11092 & ``$+$''/Halt & 7.67 & 838/162 & 50000 & 2.788 & 114617 & ``$+$''/Halt & 7.36 & 996/4\\ \hline
5000 & 2.8 & 11089 & ``$+$''/Halt & 8.02 & 791/209 & 50000 & 2.8 & 114608 & ``$+$''/Halt & 7.84 & 992/8\\ \hline
5000 & 3.0 & 11048 & ``$+$''/Halt & 12.59 & 34/966 & 50000 & 3.0 & 114461 & ``$+$''/Halt & 19.00 & 1/999\\ \hline
\end{tabular}}}
\caption{Simulation of the non-Markovian model with
$\Delta = \left\lceil \pth{\frac{p-q}{q}}n - L \sqrt{n \log n} \right\rceil$,
$p=1$, and $q=0.3$.}\label{tab:nonmarkovian-p1-larger}
\end{table}

We first focus on the special case $ p=1 $ and set $ q=0.3 $. In this case,
$ L^* \approx 2.788 $ is the critical value for the phase transition. The
simulations are listed below. As shown in Table \ref{tab:nonmarkovian-p1-larger}, when $ L<L^* $,
the opinion $ + $ wins with high frequency. In particular, opinion $ + $ wins in almost all simulations when $ L $ becomes smaller and the consensus becomes faster. On
the other hand, when $ L>L^* $, the dynamics will halt with high frequency, and
halting happens in almost all simulations as $ L $ gets larger. The rows near $L^*$ show the
finite-size transition behavior at the critical scale.

\begin{table}[h]
\centering{\scalebox{0.74}{
\begin{tabular}{|cccccc||cccccc|}
\hline
\multicolumn{6}{|c||}{\textbf{Parameter block 1}} & \multicolumn{6}{c|}{\textbf{Parameter block 2}}\\ \hline
$\bm{n}$ & $\bm{L}$ & $\bm{\Delta}$ & \textbf{Winner} & \textbf{Last day} & \textbf{Frequency} &
$\bm{n}$ & $\bm{L}$ & $\bm{\Delta}$ & \textbf{Winner} & \textbf{Last day} & \textbf{Frequency}\\ \hline
1000 & 2.0 & 501 & ``$+$''/Halt & 4.06 & 500/0 & 2000 & 2.0 & 1087 & ``$+$''/Halt & 3.98 & 100/0\\ \hline
1000 & 2.4 & 468 & ``$+$''/Halt & 8.07 & 350/150 & 2000 & 2.4 & 1038 & ``$+$''/Halt & 7.32 & 88/12\\ \hline
1000 & 2.5 & 459 & ``$+$''/Halt & 9.57 & 135/365 & 2000 & 2.5 & 1026 & ``$+$''/Halt & 10.10 & 39/61\\ \hline
1000 & 2.582 & 453 & ``$+$''/Halt & 10.78 & 45/455 & 2000 & 2.582 & 1015 & ``$+$''/Halt & 11.78 & 9/91\\ \hline
1000 & 2.6 & 451 & ``$+$''/Halt & 11.03 & 34/466 & 2000 & 2.6 & 1013 & ``$+$''/Halt & 12.38 & 8/92\\ \hline
1000 & 2.8 & 434 & ``$+$''/Halt & --- & 0/500 & 2000 & 2.8 & 989 & ``$+$''/Halt & --- & 0/100\\ \hline
1000 & 3.0 & 418 & ``$+$''/Halt & --- & 0/500 & 2000 & 3.0 & 964 & ``$+$''/Halt & --- & 0/100\\ \hline
5000 & 2.0 & 2921 & ``$+$''/Halt & 3.90 & 20/0 & 10000 & 2.0 & 6060 & ``$+$''/Halt & 3.00 & 20/0\\ \hline
5000 & 2.4 & 2839 & ``$+$''/Halt & 6.40 & 20/0 & 10000 & 2.4 & 5939 & ``$+$''/Halt & 6.25 & 20/0\\ \hline
5000 & 2.5 & 2818 & ``$+$''/Halt & 9.57 & 14/6 & 10000 & 2.5 & 5908 & ``$+$''/Halt & 8.95 & 19/1\\ \hline
5000 & 2.582 & 2801 & ``$+$''/Halt & 14.67 & 3/17 & 10000 & 2.582 & 5884 & ``$+$''/Halt & 11.50 & 6/14\\ \hline
5000 & 2.6 & 2797 & ``$+$''/Halt & 11.67 & 3/17 & 10000 & 2.6 & 5878 & ``$+$''/Halt & 15.00 & 3/17\\ \hline
5000 & 2.8 & 2756 & ``$+$''/Halt & --- & 0/20 & 10000 & 2.8 & 5817 & ``$+$''/Halt & --- & 0/20\\ \hline
5000 & 3.0 & 2715 & ``$+$''/Halt & --- & 0/20 & 10000 & 3.0 & 5757 & ``$+$''/Halt & --- & 0/20\\ \hline
\end{tabular}}}
\caption{Simulation of the non-Markovian model with
$\Delta = \left\lceil\pth{\frac{p-q}{q}}n - L \sqrt{n \log n}\right\rceil$,
$p=0.5,$ and $q=0.3$. }\label{tab:nonmarkovian-p05-larger}
\end{table}

We now simulate the case where $p=0.5$ and $q=0.3$. In this case, the critical
value for the phase transition is $ L^* \approx 2.582 $. The results of the
simulation are listed below. From Table \ref{tab:nonmarkovian-p05-larger}, we can see a similar
behavior of the outcomes as in the $p=1$ case. When $ L<L^* $, the opinion
$ + $ wins with high frequency. When $ L>L^* $, the dynamics will halt with high
frequency. The rows near $L^*$ display the finite-size transition behavior at
the critical scale.

    \section{Concluding Remarks}\label{sec:Remark}
In this paper, we introduced and analyzed two models for majority dynamics on stochastic block models. For the Markovian model, we showed that any initial bias of the opinions leads to a uniformly large advantage of the winning probabilities, and gave sufficient conditions for the leading opinion to win a.a.s. For the non-Markovian model, we analyze the phase transition between fast consensus and non-consensus behavior, which in the regimes proved here is realized by a halt of the dynamics. We identify the first and second leading-order terms that govern the critical phase. In addition, in the following we mention a few other interesting directions.

First, as  mentioned in the introduction, the $k$-majority dynamics model serves as an alternative to majority dynamics on a static random graph. It is not difficult to formulate  $k$-majority dynamics models on SBM with a community structure of those sharing the same opinion. In the literature of  $k$-majority dynamics, besides whether consensus is reached, the consensus time is also of great interest \cite{doerr2011stabilizing,becchetti2016stabilizing,cooper2015fast,
ghaffari2018nearly,shimizu2021phase}. The corresponding analogues require further careful studies.

Second, a model of a similar flavor as our non-Markovian model arises when considering the majority dynamics model for \ER graphs. Pick a constant $K>1$. Suppose that each agent alters their  opinion only if the number of neighbors holding the opposite opinion is at least $K$ times the number of neighbors holding their own opinion. Similar questions can be asked such as what difference $\Delta$ guarantees unanimity of the advantageous opinion, and how many days it would take. In this case, the block structure is  encoded locally among the agents, instead of globally on the graph.

Third, for technical reasons we have focused on the case where $p,q$ are constants in $(0,1]$ that do not depend on $n$. In certain voting models, the sparser homogeneous case $p=q\gg n^{-3/5}$ is considered; see \cite{chakraborti2021majority,benjamini2016convergence}. Extensions of our results in this more general case require further study, where we expect that the recent works of graph enumeration for sparse graphs \cite{liebenau2017asymptotic,liebenau2020asymptotic} will be applicable. Note that if $p=a(\log n)/n,\ q=b(\log n)/n$, exact recovery in the two-block stochastic block model is possible precisely above the threshold $(\sqrt{a}-\sqrt{b})^2>2$; see  \cite{abbe2015exact,mossel2014consistency}.


\bibliography{MajoritySBM}

@inproceedings{tran2020reaching,
  author={Tran, Linh and Vu, Van},
  title={{Reaching a Consensus on Random Networks: The Power of Few}},
  booktitle={Approximation, Randomization, and Combinatorial Optimization. Algorithms and Techniques (APPROX/RANDOM 2020)},
  series={Leibniz International Proceedings in Informatics (LIPIcs)},
  volume={176},
  pages={20:1--20:15},
  publisher={Schloss Dagstuhl--Leibniz-Zentrum f{\"u}r Informatik},
  year={2020},
  doi={10.4230/LIPIcs.APPROX/RANDOM.2020.20}
}

@article{TV25,
  author  = {Tran, BaoLinh and Vu, Van},
  title   = {The {``Power of Few''} Phenomenon: The Sparse Case},
  journal = {Random Structures \& Algorithms},
  volume  = {66},
  number  = {1},
  pages   = {e21260},
  year    = {2025},
  doi     = {10.1002/rsa.21260}
}

@article{kim2025new,
  title={A new density limit for unanimity in majority dynamics on random graphs},
  author={Kim, Jeong Han and Tran, BaoLinh},
  journal={arXiv preprint arXiv:2503.07447},
  year={2025}
}

@article{jaffe2025new,
  title={A new bound in Majority Dynamics on Random Graphs},
  author={Jaffe, Sean},
  journal={arXiv preprint arXiv:2503.14401},
  year={2025}
}

@article{CF25,
  author  = {Chellig, Jordan and Fountoulakis, Nikolaos},
  title   = {Majority dynamics on random graphs: the multiple states case},
  journal = {Stochastic Processes and their Applications},
  volume  = {189},
  pages   = {104682},
  year    = {2025},
  doi     = {10.1016/j.spa.2025.104682}
}

@article{CMQR23,
  author  = {Cruciani, Emilio and Mimun, Hlafo Alfie and Quattropani, Matteo and Rizzo, Sara},
  title   = {Phase transition of the $k$-majority dynamics in biased communication models},
  journal = {Distributed Computing},
  volume  = {36},
  pages   = {107--135},
  year    = {2023},
  doi     = {10.1007/s00446-023-00444-2}
}

@article{DZ25,
  author  = {d'Amore, Francesco and Ziccardi, Isabella},
  title   = {Phase transition of the 3-majority opinion dynamics with noisy interactions},
  journal = {Theoretical Computer Science},
  volume  = {1028},
  pages   = {115030},
  year    = {2025},
  doi     = {10.1016/j.tcs.2024.115030}
}

@article{zehmakan2020opinion,
  title={Opinion forming in {E}rd{\H{o}}s--{R}{\'e}nyi random graph and expanders},
  author={Zehmakan, Ahad N},
  journal={Discrete Applied Mathematics},
  volume={277},
  pages={280--290},
  year={2020},
  publisher={Elsevier}
}

@inproceedings{gartner2018majority,
  title={Majority model on random regular graphs},
  author={G{\"a}rtner, Bernd and Zehmakan, Ahad N},
  booktitle={Latin American Symposium on Theoretical Informatics},
  pages={572--583},
  year={2018},
  organization={Springer}
}

@article{shang2021note,
  title={A note on the majority dynamics in inhomogeneous random graphs},
  author={Shang, Yilun},
  journal={Results in Mathematics},
  volume={76},
  number={3},
  pages={119},
  year={2021},
  publisher={Springer}
}

@inproceedings{doerr2011stabilizing,
  title={Stabilizing consensus with the power of two choices},
  author={Doerr, Benjamin and Goldberg, Leslie Ann and Minder, Lorenz and Sauerwald, Thomas and Scheideler, Christian},
  booktitle={Proceedings of the twenty-third annual ACM symposium on Parallelism in algorithms and architectures},
  pages={149--158},
  year={2011}
}

@article{cooper2016fast,
  title={Fast plurality consensus in regular expanders},
  author={Cooper, Colin and Radzik, Tomasz and Rivera, Nicol{\'a}s and Shiraga, Takeharu},
  journal={arXiv preprint arXiv:1605.08403},
  year={2016}
}

@inproceedings{cooper2014power,
  title={The power of two choices in distributed voting},
  author={Cooper, Colin and Els{\"a}sser, Robert and Radzik, Tomasz},
  booktitle={International Colloquium on Automata, Languages, and Programming},
  pages={435--446},
  year={2014},
  organization={Springer}
}

@inproceedings{cooper2015fast,
  title={Fast consensus for voting on general expander graphs},
  author={Cooper, Colin and Els{\"a}sser, Robert and Radzik, Tomasz and Rivera, Nicolas and Shiraga, Takeharu},
  booktitle={International Symposium on Distributed Computing},
  pages={248--262},
  year={2015},
  organization={Springer}
}

@inproceedings{ghaffari2018nearly,
  title={Nearly-tight analysis for 2-choice and 3-majority consensus dynamics},
  author={Ghaffari, Mohsen and Lengler, Johannes},
  booktitle={Proceedings of the 2018 ACM Symposium on Principles of Distributed Computing},
  pages={305--313},
  year={2018}
}

@book{feller2008introduction,
  title={An Introduction to Probability Theory and Its Applications, vol 2},
  author={Feller, William},
  year={2008},
  publisher={John Wiley \& Sons}
}

@inproceedings{becchetti2016stabilizing,
  title={Stabilizing consensus with many opinions},
  author={Becchetti, Luca and Clementi, Andrea and Natale, Emanuele and Pasquale, Francesco and Trevisan, Luca},
  booktitle={Proceedings of the twenty-seventh annual ACM-SIAM symposium on Discrete algorithms},
  pages={620--635},
  year={2016},
  organization={SIAM}
}

@article{chakraborti2021majority,
  title={Majority dynamics on sparse random graphs},
  author={Chakraborti, Debsoumya and Kim, Jeong Han and Lee, Joonkyung and Tran, Tuan},
  journal={Random Structures \& Algorithms},
  volume={63},
  number={1},
  pages={171--191},
  year={2023},
  doi={10.1002/rsa.21139}
}

@inproceedings{cruciani2019distributed,
  title={Distributed community detection via metastability of the 2-choices dynamics},
  author={Cruciani, Emilio and Natale, Emanuele and Scornavacca, Giacomo},
  booktitle={Proceedings of the AAAI Conference on Artificial Intelligence},
  volume={33},
  pages={6046--6053},
  year={2019}
}

@article{shimizu2021phase,
  title={Phase transitions of Best-of-two and Best-of-three on stochastic block models},
  author={Shimizu, Nobutaka and Shiraga, Takeharu},
  journal={Random Structures \& Algorithms},
  volume={59},
  number={1},
  pages={96--140},
  year={2021},
  publisher={Wiley Online Library}
}

@article{fountoulakis2020resolution,
  title={Resolution of a conjecture on majority dynamics: Rapid stabilization in dense random graphs},
  author={Fountoulakis, Nikolaos and Kang, Mihyun and Makai, Tam{\'a}s},
  journal={Random Structures \& Algorithms},
  volume={57},
  number={4},
  pages={1134--1156},
  year={2020},
  publisher={Wiley Online Library}
}

@article{mukhopadhyay2020voter,
  title={Voter and majority dynamics with biased and stubborn agents},
  author={Mukhopadhyay, Arpan and Mazumdar, Ravi R and Roy, Rahul},
  journal={Journal of Statistical Physics},
  volume={181},
  number={4},
  pages={1239--1265},
  year={2020},
  publisher={Springer}
}

@book{ash2012information,
  title={Information Theory},
  author={Ash, Robert B},
  year={2012},
  publisher={Courier Corporation}
}

@article{chong2000ruin,
  title={The ruin problem and cover times of asymmetric random walks and {B}rownian motions},
  author={Chong, K.S. and Cowan, Richard and Holst, Lars},
  journal={Advances in Applied Probability},
  volume={32},
  number={1},
  pages={177--192},
  year={2000},
  publisher={Cambridge University Press}
}

@article{berkowitz2022central,
  title={Central limit theorem for majority dynamics: Bribing three voters suffices},
  author={Berkowitz, Ross and Devlin, Pat},
  journal={Stochastic Processes and their Applications},
  volume={146},
  pages={187--206},
  year={2022},
  publisher={Elsevier}
}

@article{clementi2015distributed,
  title={Distributed community detection in dynamic graphs},
  author={Clementi, Andrea and Di Ianni, Miriam and Gambosi, Giorgio and Natale, Emanuele and Silvestri, Riccardo},
  journal={Theoretical Computer Science},
  volume={584},
  pages={19--41},
  year={2015},
  publisher={Elsevier}
}

@inproceedings{cai2020random,
  title={Random walks on randomly evolving graphs},
  author={Cai, Leran and Sauerwald, Thomas and Zanetti, Luca},
  booktitle={International Colloquium on Structural Information and Communication Complexity},
  pages={111--128},
  year={2020},
  organization={Springer}
}

@article{becchetti2011information,
  title={Information spreading in opportunistic networks is fast},
  author={Becchetti, Luca and Clementi, Andrea and Pasquale, Francesco and Resta, Giovanni and Santi, Paolo and Silvestri, Riccardo},
  journal={CoRR abs/1107.5241},
  year={2011}
}

@article{sah2021majority,
  title={{Majority Dynamics: The Power of One}},
  author={Sah, Ashwin and Sawhney, Mehtaab},
  journal={Israel Journal of Mathematics},
  volume={267},
  number={1},
  pages={85--133},
  publisher={Springer},
  year={2025},
  doi={10.1007/s11856-024-2690-1}
}

@article{mossel2017opinion,
  title={Opinion exchange dynamics},
  author={Mossel, Elchanan and Tamuz, Omer},
  journal={Probability Surveys},
  volume={14},
  pages={155--204},
  year={2017},
  publisher={Institute of Mathematical Statistics}
}

@article{benjamini2016convergence,
  title={Convergence, unanimity and disagreement in majority dynamics on unimodular graphs and random graphs},
  author={Benjamini, Itai and Chan, Siu-On and O’Donnell, Ryan and Tamuz, Omer and Tan, Li-Yang},
  journal={Stochastic Processes and their Applications},
  volume={126},
  number={9},
  pages={2719--2733},
  year={2016},
  publisher={Elsevier}
}

@article{holland1983stochastic,
  title={Stochastic blockmodels: First steps},
  author={Holland, Paul W and Laskey, Kathryn Blackmond and Leinhardt, Samuel},
  journal={Social Networks},
  volume={5},
  number={2},
  pages={109--137},
  year={1983},
  publisher={Elsevier}
}

@article{bala1998learning,
  title={Learning from neighbours},
  author={Bala, Venkatesh and Goyal, Sanjeev},
  journal={The Review of Economic Studies},
  volume={65},
  number={3},
  pages={595--621},
  year={1998},
  publisher={Wiley-Blackwell}
}

@article{ellison1993rules,
  title={Rules of thumb for social learning},
  author={Ellison, Glenn and Fudenberg, Drew},
  journal={Journal of political Economy},
  volume={101},
  number={4},
  pages={612--643},
  year={1993},
  publisher={The University of Chicago Press}
}

@article{bollobas2007phase,
  title={The phase transition in inhomogeneous random graphs},
  author={Bollob{\'a}s, B{\'e}la and Janson, Svante and Riordan, Oliver},
  journal={Random Structures \& Algorithms},
  volume={31},
  number={1},
  pages={3--122},
  year={2007},
  publisher={Wiley Online Library}
}

@article{abbe2017community,
  title={Community detection and stochastic block models: recent developments},
  author={Abbe, Emmanuel},
  journal={Journal of Machine Learning Research},
  volume={18},
  number={177},
  pages={1--86},
  year={2018},
  publisher={JMLR.org}
}

@article{cartwright1956structural,
  title={Structural balance: a generalization of {H}eider's theory.},
  author={Cartwright, Dorwin and Harary, Frank},
  journal={Psychological Review},
  volume={63},
  number={5},
  pages={277},
  year={1956},
  publisher={American Psychological Association}
}

@article{mcculloch1943logical,
  title={A logical calculus of the ideas immanent in nervous activity},
  author={McCulloch, Warren S and Pitts, Walter},
  journal={The Bulletin of Mathematical Biophysics},
  volume={5},
  number={4},
  pages={115--133},
  year={1943},
  publisher={Springer}
}

@article{granovetter1978threshold,
  title={Threshold models of collective behavior},
  author={Granovetter, Mark},
  journal={American Journal of Sociology},
  volume={83},
  number={6},
  pages={1420--1443},
  year={1978},
  publisher={University of Chicago Press}
}

@article{nguyen2020dynamics,
  title={Dynamics of opinion formation under majority rules on complex social networks},
  author={Nguyen, Vu Xuan and Xiao, Gaoxi and Xu, Xin-Jian and Wu, Qingchu and Xia, Cheng-Yi},
  journal={Scientific Reports},
  volume={10},
  number={1},
  pages={456},
  year={2020},
  publisher={Nature Publishing Group UK London}
}

@article{mckay1997degree,
  title={The degree sequence of a random graph. {I}. The models},
  author={McKay, Brendan D and Wormald, Nicholas C},
  journal={Random Structures \& Algorithms},
  volume={11},
  number={2},
  pages={97--117},
  year={1997},
  publisher={Wiley Online Library}
}

@article{mckay1990asymptotic,
  title={Asymptotic enumeration by degree sequence of graphs of high degree},
  author={McKay, Brendan D and Wormald, Nicholas C},
  journal={European Journal of Combinatorics},
  volume={11},
  number={6},
  pages={565--580},
  year={1990},
  publisher={Elsevier}
}

@article{liebenau2020asymptotic,
  title={Asymptotic enumeration of digraphs and bipartite graphs by degree sequence},
  author={Liebenau, Anita and Wormald, Nick},
  journal={arXiv preprint arXiv:2006.15797},
  year={2020}
}

@article{liebenau2017asymptotic,
  title={Asymptotic enumeration of graphs by degree sequence, and the degree sequence of a random graph},
  author={Liebenau, Anita and Wormald, Nick},
  journal={arXiv preprint arXiv:1702.08373},
  year={2017}
}

@article{abbe2015exact,
  title={Exact recovery in the stochastic block model},
  author={Abbe, Emmanuel and Bandeira, Afonso S and Hall, Georgina},
  journal={IEEE Transactions on Information Theory},
  volume={62},
  number={1},
  pages={471--487},
  year={2016},
  publisher={IEEE},
  doi={10.1109/TIT.2015.2490670}
}

@article{mossel2014consistency,
  title={Consistency thresholds for the planted bisection model},
  author={Mossel, Elchanan and Neeman, Joe and Sly, Allan},
  journal={Electronic Journal of Probability},
  volume={21},
  pages={1--24},
  year={2016},
  doi={10.1214/16-EJP4185},
  publisher={Institute of Mathematical Statistics}
}
\bibliographystyle{plain}

\end{document}